\documentclass[a4paper,11pt,reqno]{amsart}
\usepackage{graphicx}
\usepackage{amsmath}
\usepackage[T1]{fontenc}
\usepackage{amssymb}
\usepackage{amsthm}
\usepackage{color}
\usepackage[lmargin=2.5 cm,rmargin=2.5 cm,tmargin=3.5cm,bmargin=2.5cm,
paper=a4paper]{geometry}

\newcommand{\Ab}{\mathbf A}
\newcommand{\Fb}{\mathbf F}

\newcommand{\R}{\mathbb R}
\newcommand{\Z}{\mathbb Z}
\newcommand{\N}{\mathbb N}
\newcommand{\C}{\mathbb C}

\DeclareMathOperator{\E0}{E_{\rm g}} 
\DeclareMathOperator{\IM}{Im}
\DeclareMathOperator{\curl}{curl}\DeclareMathOperator{\Div}{div}

\newtheorem{thm}{Theorem}[section]
\newtheorem{prop}[thm]{Proposition}
\newtheorem{lem}[thm]{Lemma}
\newtheorem{corol}[thm]{Corollary}
\newtheorem{theorem}[thm]{Theorem}

\newtheorem{lemma}[thm]{Lemma}

\theoremstyle{remark}
\newtheorem{rem}[thm]{Remark}

\numberwithin{equation}{section}
\title[2D Ginzburg-Landau  functional]{energy and vorticity of the Ginzburg-Landau model with variable magnetic field}
\author[K.Attar]{}
\author[]{K. Attar}
\begin{document}
\begin{abstract}
We consider the Ginzburg-Landau functional with a variable applied magnetic field in a bounded and smooth two
dimensional domain. The applied magnetic field varies smoothly and is allowed to vanish non-degenerately along a curve. Assuming that the strength of the applied magnetic field varies between two characteristic scales, and the Ginzburg-Landau parameter tends to $+\infty$, we determine an accurate asymptotic formula for the minimizing energy and show that the energy minimizers have vortices. The new aspect in the presence of a variable magnetic field is that the density of vortices in the sample is not uniform.
\end{abstract}
\maketitle
\section{Introduction}
We consider a bounded, open and simply connected set $\Omega\subset\R^2$
with smooth boundary. We suppose that $\Omega$ models a
superconducting sample subject to an applied external magnetic
field. The energy of the sample is given by the Ginzburg-Landau
functional,
\begin{multline}\label{eq-2D-GLf}
\mathcal E_{\kappa,H}(\psi,\Ab)= \int_\Omega\left( |(\nabla-i\kappa
H\Ab)\psi|^2+\frac{\kappa^2}{2}(1-|\psi|^2)^2\right)\,dx
+\kappa^2H^2\int_{\Omega}|\curl\Ab-B_0|^2\,dx\,.
\end{multline}
Here $\kappa$ and $H$ are two positive parameters, to simplify we will consider that $H=H(\kappa)$. The wave function
(order parameter) $\psi\in H^1(\Omega;\C)$ and the magnetic potential
$\Ab\in{H}^1_{\Div}(\Omega)$. The space $H^1_{\Div}(\Omega)$ is defined in \eqref{eq-2D-hs} below. Finally, the function $B_{0}\in C^{\infty}(\overline{\Omega})$ gives the intensity of the external variable magnetic field. Let $\Gamma=\{x\in\overline{\Omega}, B_{0}(x)=0\}$, then, we assume that $B_{0}$ satisfies :
\begin{equation}\label{B(x)}
\left\{
\begin{array}{ll}
|B_{0}| + |\nabla B_0 | >0&\mbox{ in } \overline{\Omega}\\
\nabla B_{0}\cdot \vec{n}\neq 0 &\mbox{on}~ \Gamma\cap\partial\Omega\,.
\end{array}
\right.
\end{equation}

The assumption in \eqref{B(x)} implies that for any open set $\omega$ relatively compact in $\Omega$  the set $\Gamma\cap\omega$ will be either empty, or consists of a union of smooth curves. Here, the definition of the functional \eqref{eq-2D-GLf} is taken as in \cite{SB}. In \cite{SS}, the scaling for the intensity of the external magnetic field (denoted by $h$) is different. We choose the scaling from \cite{SB} for convenience when estimating the ground state energy of the functional.

Let $\Fb:\Omega\rightarrow\R^{2}$ be the unique vector field such that,
\begin{equation}\label{div-curlF}
\Div \Fb=0\,{\rm~and~}\,{\rm\curl \Fb}=B_{0}~{\rm
in~\Omega}\,,\,\,\,~\nu\cdot\Fb=0~{\rm on}~\partial\Omega.
\end{equation}
The vector $\nu$ is the unit interior normal vector of
$\partial\Omega$. We define the
space,
\begin{equation}\label{eq-2D-hs} H^1_{\Div}(\Omega)=\{\Ab=(\Ab_{1},\Ab_{2})\in
H^1(\Omega)^{2}~:~\Div \Ab=0~{\rm in}~\Omega \,,\,\Ab\cdot\nu=0~{\rm
on}\,
\partial\Omega \,\}.
\end{equation}
Critical points $(\psi,\Ab)\in H^1(\Omega;\C)\times
H^1_{\Div}(\Omega)$ of $\mathcal E_{\kappa,H}$ are weak solutions of
the Ginzburg-Landau equations,
\begin{equation}\label{eq-2D-GLeq}
\left\{
\begin{array}{llll}
-(\nabla-i\kappa H\Ab)^2\psi=\kappa^2(1-|\psi|^2)\psi&{\rm in}&
\Omega
\\
-\nabla^{\bot}\curl(\Ab-\Fb)=\displaystyle\frac1{\kappa
H}\IM(\overline{\psi}\,(\nabla-i\kappa
H\Ab)\psi) &{\rm in}& \Omega\\
\nu\cdot(\nabla-i\kappa H\Ab)\psi=0&{\rm
on}&\partial\Omega\\
\curl\Ab=\curl\Fb&{\rm on}&\partial\Omega\,.
\end{array}\right.
\end{equation}
Here, $\curl\Ab=\partial_{x_1}\Ab_{2}-\partial_{x_2}\Ab_{1}$ and
$\nabla^{\bot}\curl\Ab=(\partial_{x_2}(\curl\Ab),
-\partial_{x_1}(\curl\Ab)).$

For a solution $(\psi,\Ab)$ of \eqref{eq-2D-GLeq}, the function $\psi$ describes the superconducting properties of the material and $(\kappa H\curl\Ab)$ is the induced magnetic field. The number $\kappa$ is a parameter describing the properties of the material, and the number $H$ measures the variation of the intensity of the applied magnetic field. We focus on the regime of large values of $\kappa$, $\kappa\rightarrow +\infty$.

In this paper, we study the ground state energy defined as follows:
\begin{equation}\label{eq-2D-gs}
\E0(\kappa,H)=\inf\big\{ \mathcal
E_{\kappa,H}(\psi,\Ab)~:~(\psi,\Ab)\in H^1(\Omega;\C)\times
H^1_{\Div}(\Omega)\big\}\,. 
\end{equation}
More precisely, we give an asymptotic estimate valid when $H(\kappa)$ satisfies: 
\begin{equation}\label{cond-H}
C_{\text{min}}\kappa^{\frac{1}{3}}\leq H(\kappa)\ll\kappa\qquad{\rm as}~\kappa\longrightarrow+\infty\,,
\end{equation}
where $C_{\text{min}}$ is a positive constant.\\
The behavior of $\E0(\kappa,H)$ involves a  function $\hat{f}:[0,1]\longrightarrow[0,\frac{1}{2}]$ introduced in \eqref{f(b)} below. The function $\hat{f}$ is increasing, continuous and $\hat{f}(b)=\frac{1}{2}$, for all $b\geq 1$.

Under the assumption that $B_{0}(x)$ satisfies \eqref{B(x)} and that the function $H=H(\kappa)$ satisfies
\begin{equation}\label{old-condH}
C_{1}\kappa\leq H\leq C_{2}\kappa\,,
\end{equation}
where $C_{1}$ and $C_{2}$ are positive constants, we obtained \footnote{After a change of notation} in \cite{KA} that 
\begin{equation}\label{old-eq}
\E0(\kappa,H)=\kappa^{2}\int_{\Omega}\hat{f}\left(\frac{H}{\kappa}|B_{0}(x)|\right)\,dx+\textit{o}\left(\kappa H\right)\,,\qquad{\rm as}~\kappa\longrightarrow+\infty\,.
\end{equation}
In this paper, we generalize this result to the case when $H(\kappa)$ satisfies \eqref{cond-H}.
%%%%%%%%%%%%%%%%%%%%%%%%%%%%%%%%%%%%%%TTTTTTTTTTTTTT%%%%%%%%%%%%%%%%%%%%%%%%%%%%%%%%%%%%%%%%%%%%%%%%%%%%%%
\begin{thm}\label{thm-2D-main}
Under Assumptions \eqref{B(x)} and  \eqref{cond-H},  the ground state energy in \eqref{eq-2D-gs} satisfies, as $\kappa\longrightarrow +\infty$
\begin{equation}\label{eq-2D-thm}
 \E0(\kappa,H)=\kappa^{2}\int_{\Omega}  \hat{f}\left(\frac{H}{\kappa}|B_{0}(x)|\right)\,dx+\textit{o}\left(\kappa H\ln\frac{\kappa}{H}\right)\,.
\end{equation}
\end{thm}
%%%%%%%%%%%%%%%%%%%%%%%%%%%%%%%%%%%%%%%%%%%%%%%%%%%%%%%%%%%%%%%%%%%%%%%%%%%%%%%%%%%%%%%%%%%%%%%%%%%%%%%%
We will see in Remark~\ref{R1} that the second term in the right hand side of  \eqref{eq-2D-thm}, which is actually more simply $\textit{o} (\kappa H \ln \kappa)$ when \eqref{cond-H} is satisfied, is of lower order compared with the leading term. Actually (see in Theorem~\ref{pro-f(b)}), the function  $\hat{f}$ satisfies
$$
\hat{f}(b)=\frac{b}{2}\ln \frac{1}{b}\,(1+\hat{s}(b))\,,\qquad {\rm as} ~b\longrightarrow 0\,,
$$
with $\hat{s}(b)=\textit{o}(1)$.\\
As a consequence of the behaviour of $\hat{f}$ above, \eqref{eq-2D-thm} becomes
\begin{equation}\label{sr-eq-2D-thm} 
 \E0(\kappa,H)= \frac 12 \kappa H\left[\int_{\Omega}|B_{0}(x)|\ln \frac{\kappa}{H |B_{0}(x)|}\,dx\right]\,(1+\textit{o}(1))\,.
\end{equation}
When the magnetic field is constant (i.e $B_{0}$ is a constant function), \eqref{sr-eq-2D-thm} is proved in \cite{SS3} under the relaxed condition
\begin{equation}\label{sr-cond-H}
\frac{\ln\kappa}{\kappa}\ll H\ll\kappa\,.
\end{equation}

The reason why we do not obtain \eqref{sr-eq-2D-thm} under the relaxed condition \eqref{sr-cond-H} is probably technical. The method is to construct test configurations with a Dirichlet boundary condition. We can not construct periodic configurations as in \cite{SS3} because the magnetic field $B_{0}$ is variable. The approach used in the proof of Theorem~\ref{thm-2D-main} is close to that in \cite{AK} which studies the same problem when $\Omega\subset\R^{3}$ and $B_{0}$ is constant.
\begin{rem}
Notice that when $H(\kappa)$ satisfies \eqref{old-condH} we have
$$\textit{o}(\kappa H)=\textit{o}\left(\kappa H\left(\Big|\ln\frac{H}{\kappa}\Big|+1\right)\right)\,.$$
\end{rem}

If we assume that there exist positive constants $C_{\rm min}$ and $C_{1}$ and $H(\kappa)$ satisfies
\begin{equation}
C_{\rm min} \kappa^{\frac{1}{3}}\leq H(\kappa)\leq C_{1} \kappa\,,
\end{equation}  
then \eqref{old-eq} and \eqref{eq-2D-thm} can be rewritten in a unique statement:
\begin{equation}\label{eq:1.14}
 \E0(\kappa,H)=\kappa^{2}\int_{\Omega}  \hat{f}\left(\frac{H}{\kappa}|B_{0}(x)|\right)\,dx+\textit{o}\left(\kappa H\left(\Big|\ln\frac{H}{\kappa}\Big|+1\right)\right)\,.
\end{equation}
\begin{rem}
When the set $\Gamma=\{x\in\overline{\Omega}, B_{0}(x)=0\}$ consists of a finite number of smooth curves and the intensity of the magnetic field $H$ satisfies $\kappa\ll H\leq \mathcal{O}(\kappa^{2})$, then the energy $\E0(\kappa,H)$ in \eqref{eq-2D-GLf} is estimated in \cite{HK}.
\end{rem}
%%%%%%%%%%%%%%%%%%%%%%%%%%%%%%%%%%%%%%%%%%CCCCCCCCCCCCCCCCCC%%%%%%%%%%%%%%%%%%%%%%%%%%%%%%%%%%%%%%%%%%%%
Theorem~\ref{thm-2D-main} admits the following 
corollary which is useful in the proof of Theorem~\ref{th-loc-enrgy} below.
The content of Corollary~\ref{estimate-magnetic-energy} gives us that the magnetic energy is  small compared with the leading term in \eqref{eq:1.14}.

\begin{corol}\label{estimate-magnetic-energy}
Suppose that the assumptions of Theorem~\ref{thm-2D-main} hold. Then, the magnetic energy of the minimizer satisfies
\begin{equation}\label{eq-estimate-magnetic-energy}
(\kappa H)^{2}\int_{\Omega}|\curl\Ab-B_{0}|^{2}\,dx=\textit{o}\left(\kappa H \ln\frac{\kappa}{H}\right)\,,\qquad\text{as}~\kappa\longrightarrow+\infty\,.
\end{equation}
\end{corol}
%%%%%%%%%%%%%%%%%%%%%%%%%%%%%%%%%%%%%%%%%%%%%%%%%%%%%%%%%%%%%%%%%%%%%%%%%%%%%%%%%%%%%%%%%%%%%%%%%%%%%%%%%%
If $(\psi,\Ab)\in H^1(\Omega;\C)\times H^1_{\Div}(\Omega)$, we
introduce the energy density,
$$e(\psi,\Ab)=|(\nabla -i\kappa
H\Ab)\psi|^2+\frac{\kappa^2}2(1-|\psi|^{2})^2\,.$$ We also
introduce the local energy of $(\psi,\Ab)$ in
a domain $\overline{\mathcal{D}}\subset\Omega$ :
\begin{equation}\label{eq-GLe0}
\mathcal E_0(\psi,\Ab; \mathcal{D})=\int_{\mathcal{D}}e(\psi,\Ab)\,dx\,.
\end{equation}
Furthermore, we define the Ginzburg-Landau energy of $(\psi,\Ab)$ in
a domain $\overline{\mathcal{D}}\subset\Omega$ as follows,
\begin{equation}\label{eq-GLen-D}
\mathcal E(\psi,\Ab;\mathcal{D})=\mathcal E_0(u,\Ab; \mathcal{D})+(\kappa H)^2\int_{\Omega}|\curl(\Ab-\Fb)|^2\,dx\,.
\end{equation}
If $\mathcal{D}=\Omega$, we sometimes omit the dependence on the
domain and write $\mathcal E_0(\psi,\Ab)$ for $\mathcal
E_0(\psi,\Ab;\Omega)$.

The next theorem gives a local version of Theorem~\ref{thm-2D-main}.
%%%%%%%%%%%%%%%%%%%%%%%%%%%%%%%%%%TTTTTTTTTTTTTTTTT%%%%%%%%%%%%%%%%%%%%%%%%%%%%%%%%%%%%%%%%%%%%%%%%%%%%%%
 \begin{theorem}\label{th-loc-enrgy}
Under Assumption \eqref{B(x)}, if $(\psi,\Ab)$ is a minimizer of \eqref{eq-2D-GLf} and $\mathcal{D}$ is regular set such that $\overline{\mathcal{D}}\subset \Omega$, then the following is true.
\begin{enumerate}
\item If $H(\kappa)$ satisfies \eqref{cond-H}, then,
\begin{equation}\label{lw-loc-enrgy}
\mathcal{E}(\psi,\Ab, \mathcal{D})\geq \kappa^{2}\int_{\mathcal{D}}\hat{f}\left(\frac{H}{\kappa}|B_{0}(x)|\right)\,dx+\textit{o}\left( \kappa H \ln\frac{\kappa}{H}\right)\,,\quad\text{as}~\kappa\longrightarrow+\infty\,.
\end{equation}
\item If $H(\kappa)$ satisfies
\begin{equation}\label{cond-H2}
C^{1}_{\text{min}}\kappa^{\frac{3}{5}}\leq H\ll\kappa\qquad{\rm as}~\kappa\longrightarrow+\infty\,,
\end{equation}
where $C^{1}_{\text{min}}$ is a positive constant, then
\begin{equation}\label{up-loc-enrgy}
\mathcal{E}(\psi,\Ab, \mathcal{D})\leq \kappa^{2}\int_{\mathcal{D}}\hat{f}\left(\frac{H}{\kappa}|B_{0}(x)|\right)\,dx+\textit{o}\left( \kappa H \ln\frac{\kappa}{H}\right)\,,\quad\text{as}~\kappa\longrightarrow+\infty\,.
\end{equation}
\end{enumerate}
\end{theorem}
As a consequence of the proof of Theorem~\ref{thm-2D-main}, the methods used in \cite{SS3} allow us to obtain  information regarding the distribution of vortices in $\Omega$. When the magnetic field is constant (i.e $B_{0}$ is a constant), it is proved in \cite{SS3} that $\psi$ has vortices whose density tends to be uniform.  In Section~\ref{section7} we will  prove that, if $(\psi,\Ab)$ is a minimizer of \eqref{eq-2D-GLf} and $B_{0}(x)$ is a variable magnetic field, then, $\psi$ has vortices that are distributed everywhere in $\Omega$ but with a non uniform density.

The next theorem was proved by E. Sandier and S. Serfaty in \cite{SS3} when the magnetic field is constant $(B_{0}(x)=1).$
%%%%%%%%%%%%%%%%%%%%%%%%%%%%%%%%%%%%%%%%%TTTTTTTTTTTTTTTT%%%%%%%%%%%%%%%%%%%%%%%%%%%%%%%%%%%%%%%%%%%%%%%%
\begin{thm}\label{distribution-vortices}
Suppose that Assumption \eqref{B(x)} holds and that $H(\kappa)$ satisfies \eqref{cond-H}. Let $(\psi,\Ab)$ be a minimizer of \eqref{eq-2D-GLf}. Then there exists $m=m(\kappa)$  disjoint disks $\left(D_{i}(a_{i},r_{i})\right)_{i=1}^{m}$  in $\Omega$ such that, as $\kappa\longrightarrow+\infty\,$,
\begin{enumerate}
\item \label{p1} $\sum_{i=1}^m  r_{i}\leq (\kappa H)^{\frac{1}{2}}\,\left(\ln\frac{\kappa}{H}\right)^{-\frac{7}{4}}\,\int_{\Omega}\frac{1}{\sqrt{|B_{0}(x)|}}\,dx\,(1+\textit{o}(1))\,.$
\item $|\psi|\geq\frac{1}{2}$ on $\cup_{i}\partial D_{i}$\,.
\item \label{p2} If $d_{i}={\rm deg}\left(\frac{\psi}{|\psi|}, \partial D_{i}\right)$ is the winding number of $\frac{\psi}{|\psi|}$ on $\partial D_{i}$, then as $\kappa\longrightarrow+\infty$
\end{enumerate}
$$\mu_{\kappa}=\frac{2\pi}{\kappa H}\sum_{i=1}^{m}d_{i}\delta_{a_{i}}\longrightarrow B_{0}(x)\,dx\qquad{\rm and}\qquad|\mu_{\kappa}|=\frac{2\pi}{\kappa H}\sum_{i=1}^{m}|d_{i}|\delta_{a_{i}}\longrightarrow |B_{0}(x)|\,dx\,,$$
in the weak sense of measures \footnote{$\mu_{\kappa}$ converge weakly to $\mu$ means that:~\\
$$
\mu_{\kappa}(f)\longrightarrow\mu(f)\,,\qquad\forall f\in C_{0}(\Omega)\,.
$$
}, where $dx$ is the Lebesgue measure on $\R^{2}$ restricted to $\Omega\,$.
\end{thm}
The measure $\mu$ describes the distribution of vortices see Fig.1, and it is called the \textit{vorticity measure}, the function $\textit{o}(1)$ is bounded independently of the choice of the minimizer $(\psi,\Ab)$.
%%%%%%%%%%%%%%%%%%%%%%%%%%%%%%%%%%%%%%%%NNNNNNNNNNNNNNN%%%%%%%%%%%%%%%%%%%%%%%%%%%%%%%%%%%%%%%%%%%%%%%
\subsection* {Notation.}
{Throughout the paper, we use the following notation:}
\begin{itemize}
%\item The letter $C$ denotes a positive constant that is independent of the parameters $\kappa$ and $H$, and whose value  may change from a formula to another.
\item If $a(\kappa)$ and $b(\kappa)$ are two positive functions, we write $a(\kappa)\ll b(\kappa)$ if $a(\kappa)/b(\kappa)\to0$ as $\kappa\to\infty$.
\item If $a(\kappa)$ and $b(\kappa)$ are two functions with $b(\kappa)\not=0$, we write $a(\kappa)\sim b(\kappa)$ if $a(\kappa)/b(\kappa)\to1$ as $\kappa\to\infty$.
\item If $a(\kappa)$ and $b(\kappa)$ are two positive functions, we write $a(\kappa)\approx b(\kappa)$
if there exist positive constants $c_1$, $c_2$ and $\kappa_0$ such
that $c_1b(\kappa)\leq a(\kappa)\leq c_2b(\kappa)$ for all
$\kappa\geq\kappa_0$.
\item Given $R>0$ and $x=(x_{1},x_{2})\in\R^{2}$, $Q_{R}(x)=(-R/2+x_{1},R/2+x_{1})\times(-R/2+x_{2},R/2+x_{2})$ denotes the
square of side length $R$ centered at $x$ and we write $Q_{R}=Q_{R}(0)$.
\end{itemize}

\section{A reference problem}
Consider two constants $b\in (0,1)$ and $R>0$. If $u\in H^{1}(Q_{R})$, we define the following Ginzburg-Landau energy,
\begin{equation}\label{eq-GL-F}
F^{\sigma}_{b,Q_{R}}(u)=\int_{Q_{R}}\left(b|(\nabla-i\sigma\Ab_0)u|^2+\frac{1}{2}\left(1-|u|^2\right)^{2}\right)\,dx\,,
\end{equation}
where $\sigma\in\{-1,+1\}$ and
\begin{equation}\label{eq-hc2-mpA0}
\Ab_{0}(x)=\frac1{2}(-x_2,x_1)\,,\qquad\forall\,x=(x_1,x_2)\in\R^{2}\,.
\end{equation}
Notice that the magnetic potential $\Ab_0$ satisfies:
$$\curl\Ab_0=1\,\,{\rm~in~}\R^{2}\,.$$
We introduce the two ground state energies
\begin{eqnarray}
e_{N}(b,R)=\inf\left\{F^{+1}_{b,Q_{R}}(u): u\in H^{1}(Q_{R};\C)\right\}\label{eN}\\
e_{D}(b,R)=\inf\left\{F^{+1}_{b,Q_{R}}(u): u\in H^{1}_{0}(Q_{R};\C)\right\}\label{eD}\,.
\end{eqnarray}
The minimization of the functional $F^{+1}_{b,Q_{R}}$ over `\textit{magnetic periodic}' functions appears naturally in the proof. Let us introduce the following space
\begin{equation}\label{ER}
E_{R}=\left\{u\in H^{1}_{\rm loc}(\R^{2};\C): u(x_{1}+R,x_{2})=e^{iR\frac{x_{2}}{2}}u(x_1,x_2), u(x_{1},x_{2}+R)=e^{-iR\frac{x_{1}}{2}}u(x_1,x_2)\right\}\,,
\end{equation}
together with the ground state energy
\begin{equation}\label{ep}
e_{p}(b,R)=\inf\left\{F^{+1}_{b,Q_{R}}(u): u\in E_{R}\right\}\,.
\end{equation}
Since $F^{+1}_{b,Q_{R}}$ is bounded from below, there exists for each $e_{\#}(b,R)$ with $\#\in\{N, D, p\}$, a ground state (minimizer). Note also that by comparison of the three domains of minimization it is clear that 
\begin{equation}\label{eN<ep<eD}
e_{N}(b,R)\leq e_{p}(b,R)\leq e_{D}(b,R)\,.
\end{equation}

In the three cases, if $u$ is such a ground state, $u$ satisfies the Ginzburg-Landau equation 
$$b(\nabla-i\Ab_{0})^{2}u=(1-|u|^{2})u\,,$$
 and it results from a standard application of the maximum principle that
\begin{equation}\label{up-u-1}
|u|\leq 1\,.
\end{equation} 
As $F^{+1}_{b,Q_{R}}(u)=F^{-1}_{b,Q_{R}}(\overline{u})$, it is also immediate that,
\begin{equation}\label{nE=pE}
\inf_{u\in H^{1}(Q_{R};\C)}F^{+1}_{b,Q_{R}}(u)=\inf_{u\in H^{1}(Q_{R};\C)}F^{-1}_{b,Q_{R}}(u)\,.
\end{equation}
%%%%%%%%%%%%%%%%%%%%%%%%%%%%%%%%%%%%%%%%LIMITING FUNCTIONS%%%%%%%%%%%%%%%%%%%%%%%%%%%%%%%%%%%%%%%%%%%%%%%%%%

In the next theorem we will define the limiting function $\hat{f}$, which describes the ground state energy of both two and three dimensional superconductors subject to high magnetic fields (see \cite{FK2}).
\begin{theorem}\label{pro-f(b)}
Let $e_{p}(b,R)$ be as introduced in \eqref{ep}.
\begin{enumerate}
\item For any $b\in [0,\infty)$, there exists a constant $\hat{f}(b)\geq 0$ such that
\begin{equation}\label{f(b)}
\hat{f}(b)=\lim_{R\longrightarrow\infty}\frac{e_{p}(b,R)}{|Q_{R}|}=\lim_{R\longrightarrow\infty}\frac{e_{D}(b,R)}{|Q_{R}|}\,.
\end{equation}
\item For all $b\geq 1$, $\hat{f}(b)=\frac{1}{2}$.
\item The function $[0,\infty)\ni b\longmapsto \hat{f}(b)$ is continuous, non-decreasing and its range is the interval $[0,1/2]$.
\item As $b\longrightarrow 0_{+}$, $\hat{f}(b)$ satisfies
\begin{equation}\label{main-f(b)}
\hat{f}(b)=\frac{b}{2}\ln\frac{1}{b}\,(1+\hat{s}(b))\,,
\end{equation}
where the function $\hat{s}:(0,+\infty)\longmapsto(-\infty,+\infty)$ satisfies
$$
\lim_{b\longrightarrow 0} \hat{s}(b)\longrightarrow 0\,.
$$
\item There exist universal constants $C$ and $R_{0}$ such that
\begin{equation}\label{est-f(b)}
\forall R\geq R_{0},\qquad\forall b\in[0,1],\qquad\left|\hat{f}(b)-\frac{e_{p}(b,R)}{R^{2}}\right|\leq\frac{C}{R}.
\end{equation}
\item\label{p6} There exist positive constants $b_{0}$, $R_{0}$ and a function 
\begin{equation}\label{err}
{\rm err}: (0,1)\times(0,+\infty)\longrightarrow(0,+\infty)\,,
\end{equation}
such that
\begin{equation}
\forall \epsilon\geq0\,,\exists \eta\geq0~{\rm if}~|b|+\frac{1}{R}<\eta~{\rm then}~|{\rm }err(b,R)|\leq \epsilon\,,
\end{equation}
and
\begin{equation}\label{eN<f}
\forall b\in(0,b_{0})\,,\qquad\forall R\in(R_{0},+\infty)
\,,\qquad \frac{e_{N}(b,R)}{R^{2}}\geq\hat{f}(b)(1-{\rm err}(b,R))\,.\end{equation}
\end{enumerate}
\end{theorem}
The limiting function $\hat{f}$ was defined in (\cite{AS}, \cite{SS2}, \cite{AK}). The estimate in \eqref{main-f(b)} and \eqref{est-f(b)} are obtained by Fournais-Kachmar (see \cite[Theorem~2.1~and~Proposition~2.8]{FK2}) and by Kachmar (see \cite[Theorem~2.4]{AK}) respectively. The lower bound in \eqref{eN<f} is a consequence of \cite[Theorem~2.1 and (2.9)]{AK}.

We need the next proposition in the proof of the lower bound of $\mathcal{E}_{\kappa,H}(\psi,\Ab)$.
%%%%%%%%%%%%%%%%%%%%%%%%%%%%%%%%%%%%%%%%%%%%%%%%LLLLL%%%%%%%%%%%%%%%%%%%%%%%%%%%%%%%%%%%%%%%%%%%%%%%%%%%%
\begin{prop}
There exists a positive constant $C$, such that if 
\begin{equation}\label{cond-b&R}
R\geq 1\qquad{\rm and}\qquad 0<b< 1\,,
\end{equation}
then,
\begin{eqnarray}
e_{D}(b,R)\leq e_{N}(b,R)+CR\,b^{\frac{1}{2}}\label{eD<ep}\,.
\end{eqnarray}
\end{prop}

\begin{proof}
Without loss of generality, we can suppose $\sigma=+1$. Let $u\in H^{1}(Q_{R})$ be a minimizer of the functional in \eqref{eq-GL-F}, i.e. such that:
\begin{equation}\label{eN=F}
e_{N}(b,R)=F^{+1}_{b,Q_{R}}(u)=\int_{Q_{R}}\left(b|(\nabla-i\Ab_0)u|^2+\frac{1}{2}\left(1-|u|^2\right)^{2}\right)\,dx\,.
\end{equation}
We introduce a cut-off function $\chi_{R,b}\in C_{c}^{\infty}(\R^{2})$ such that
\begin{equation}\label{def-chiR}
0\leq\chi_{R,b}\leq 1\qquad {\rm in}~\R^{2}\,,\qquad {\rm supp}\chi_{R,b}\subset Q_{R}\,,\qquad\chi_{R,b}=1\qquad{\rm in}~Q_{R-b^{\frac{1}{2}}}\,.
\end{equation}
In addition, the function $\chi_{R,b}$ can be chosen such that for some universal constants $C$ and $C'$, we have,
\begin{equation}\label{nabla-delta}
|\nabla\chi_{R,b}|\leq Cb^{-\frac{1}{2}}\qquad {\rm and} \qquad|\Delta\chi_{R,b}|\leq C'b^{-1}\,,\qquad\forall R\geq 1\qquad{\rm and}\qquad\forall b\in(0,1)\,.
\end{equation}
Let $u_{R,b}(x)=\chi_{R,b}(x)u(x)$. Then $u_{R,b}\in H^{1}_{0}(Q_{R})$ and consequently
\begin{equation}\label{eD=F}
e_{D}(b,R)\leq F^{+1}_{b,Q_{R}}(u_{R,b})\,.
\end{equation}
We rewrite $F^{+1}_{b,Q_{R}}(u_{R,b})$ as follows,
\begin{align}\label{1st-F=}
F^{+1}_{b,Q_{R}}(u_{R,b})&=\int_{Q_{R}} \left( b|(\nabla-i\Ab_{0})\chi_{R,b}u|^{2}+\frac{1}{2}\left(1-|\chi_{R,b}u|^{2}\right)^{2}\right)\,dx\nonumber\\
&=\int_{Q_{R}} \left( b|(\nabla-i\Ab_{0})\chi_{R,b}u|^{2}+\frac{1}{2}\Big(1-2|u|^{2}+|\chi_{R,b}u|^{4}+2(|u|^{2}-|\chi_{R,b}u|^{2})\Big)\right)\,dx\nonumber\\
&\leq\int_{Q_{R}} \left( b|(\nabla-i\Ab_{0})\chi_{R,b}u|^{2}+\frac{1}{2}\Big(1-|u|^{2}\Big)^{2}\right)\,dx+\int_{Q_{R}\setminus Q_{R-b^{\frac{1}{2}}}}\Big(1-|\chi_{R,b}|^{2}\Big)|u|^{2}\,dx\,.
\end{align}
We estimate from above the term $\displaystyle\int_{Q_{R}}|(\nabla-i\Ab_{0})\chi_{R,b}u|^{2}\,dx$ as follows:
\begin{align*}
\int_{Q_{R}}|(\nabla-i\Ab_{0})\chi_{R,b}u|^{2}\,dx&=\Big\langle(\nabla-i\Ab_{0})\chi_{R,b}u,(\nabla-i\Ab_{0})\chi_{R,b}u\Big\rangle\\
&=\Big\langle\nabla\chi_{R,b}u+\chi_{R,b}(\nabla-i\Ab_{0})u,\nabla\chi_{R,b}u+\chi_{R,b}(\nabla-i\Ab_{0})u\Big\rangle\\
&=\Big\langle\nabla\chi_{R,b}u,\nabla\chi_{R,b}u\Big\rangle+\Big\langle\chi_{R,b}(\nabla-i\Ab_{0})u,\chi_{R,b}(\nabla-i\Ab_{0})u\Big\rangle\\
&\qquad\qquad+\Big\langle\nabla\chi_{R,b}u,\chi_{R,b}(\nabla-i\Ab_{0})u\Big\rangle+\Big\langle\chi_{R,b}(\nabla-i\Ab_{0})u,\nabla\chi_{R,b}u\Big\rangle\,.
\end{align*}
An integration by parts yields,
\begin{multline}
\langle\nabla\chi_{R,b}u,\chi_{R,b}(\nabla-i\Ab_{0})u\rangle=-\langle\nabla\chi_{R,b}u,\nabla\chi_{R,b}u\rangle\\
-\langle\chi_{R,b}\Delta\chi_{R,b}u,u\rangle-\langle\chi_{R,b}(\nabla-i\Ab_{0})u,\nabla\chi_{R,b}u\rangle\,,
\end{multline}
which implies that
\begin{equation}\label{2nd-F=}
\int_{Q_{R}}|(\nabla-i\Ab_{0})\chi_{R,b}u|^{2}\,dx=\Big\langle\chi_{R,b}(\nabla-i\Ab_{0})u,\chi_{R,b}(\nabla-i\Ab_{0})u\Big\rangle-\Big\langle\chi_{R,b}\Delta\chi_{R,b}u,u\Big\rangle\,.
\end{equation}
Putting \eqref{2nd-F=} into \eqref{1st-F=}, we get
\begin{align}
F^{+1}_{b,Q_{R}}(u_{R,b})&\leq\int_{Q_{R}} \left( b|\chi_{R,b}(\nabla-i\Ab_{0})u|^{2}+\frac{1}{2}\Big(1-|u|^{2}\Big)^{2}\right)\,dx\nonumber\\
&\qquad\qquad\qquad+\int_{Q_{R}\setminus Q_{R-b^{\frac{1}{2}}}}\Big(1-|\chi_{R,b}|^{2}\Big)|u|^{2}\,dx+b\int_{Q_{R}\setminus Q_{R-b^{\frac{1}{2}}}}|\Delta\chi_{R,b}|\,|u|^{2}\,dx\,.
\end{align}
By using the bound $|u|\leq 1$, \eqref{nabla-delta} and the assumption on the support of $\chi_{R,b}$ in \eqref{def-chiR}, it is easy to check that,
$$F^{+1}_{b,Q_{R}}(u_{R,b})\leq F^{+1}_{b,Q_{R}}(u)+CR\,b^{\frac{1}{2}}\,.$$
Using \eqref{eD=F} and \eqref{eN=F}, we  get
$$e_{D}(b,R)\leq e_{N}(b,R)+CR\,b^{\frac{1}{2}}\,.$$
\end{proof}
%%%%%%%%%%%%%%%%%%%%%%%%%%%%%%%%%%%%%%%%%CCCCCC%%%%%%%%%%%%%%%%%%%%%%%%%%%%%%%%%%%%%%%%%%%%%%%%%%%%%%%%%%
\begin{corol}\label{limite}
With $\hat{f}(b)$ introduced in \eqref{f(b)}, it holds,
\begin{equation}
\hat{f}(b)=\lim_{R\longrightarrow+\infty}\frac{e_{N}(b,R)}{R^{2}}\,.
\end{equation}
\end{corol}
\begin{proof}
We have from \eqref{eN<ep<eD} and  \eqref{eD<ep} that, for any $ b\in(0,1)\,,$
$$ e_{D}(b,R)-CR\,b^{\frac{1}{2}}\leq e_{N}(b,R)\leq e_{D}(b,R)\,.$$
Having in mind \eqref{f(b)}, we divide all sides of this inequality by $R^{2}$ and then take the limit as $R\longrightarrow+\infty$ . That gives us 
$$\hat{f}(b)=\lim_{R\longrightarrow\infty}\frac{e_{N}(b,R)}{R^{2}}\,.$$
\end{proof}
%%%%%%%%%%%%%%%%%%%%%%%%%%%%%%%%%%%%%%%%%%%PPPPPPP%%%%%%%%%%%%%%%%%%%%%%%%%%%%%%%%%%%%%%%%%%%%%%%%%%%%%%%%
\begin{prop}[Fournais]\label{FR}
There exists a positive constant $C$, such that if \eqref{cond-b&R} is satisfied, then
\begin{eqnarray}
\frac{e_{D}(b,R)}{R^{2}}\leq \hat{f}(b)+C\frac{\sqrt{b}}{R}\label{eD<f(b)}\,,\\
\frac{e_{D}(b,R)}{R^{2}}\geq \hat{f}(b)\label{eD>f(b)}\,.
\end{eqnarray}
\end{prop}
\begin{proof}
We have already seen that
\begin{equation}\label{f1}
\hat{f}(b) = \lim_{R\rightarrow +\infty} \frac{e_D(b,R)}{R^2}\,.
\end{equation}
\textbf{Let us first prove \eqref{eD>f(b)}}. Let $n\in\N^{*}$ and $R>0$. Let $u\in H^{1}_{0}(Q_{R})$ be a minimizer of $F^{+1}_{b,Q_{R}}$ (i.e. $e_{D}(b,R)=F^{+1}_{b,Q_{R}}(u)$). We extend $u$ to a function $\widetilde{u}\in H^{1}_{0}(Q_{nR})$ by `\textit{magnetic periodicity}' as follows 
$$\widetilde{u}(x_{1}+R,x_{2})=e^{iR\frac{x_{2}}{2}}u(x_{1},x_{2})\,,\qquad\widetilde{u}(x_{1},x_{2}+R)=e^{-iR\frac{x_{1}}{2}}u(x_{1},x_{2})\,.$$
Let $\mathcal J^n =\{j\in\Z, 1\leq j\leq n^{2}\}$. Notice that, the square $Q_{nR}$ is formed exactly of $n^{2}$ squares $(\overline{Q_{R}(x_{0}^{j})})_{j\in \mathcal J^n }$. We define in each $Q_{R}(x_{0}^{j})$ the following function
$$u_{j}=\widetilde{u}_{\mid_{Q_{R}(x_{0}^{j})}}\,.$$
Observe that $u_{j}$ is a minimizer of $F^{+1}_{b,Q_{R}(x_{0}^{j})}$ in $H^{1}_{0}(Q_{R}(x_{0}^{j}))$ and if we extend $u_{j}$ by $0$ outside of $Q_{R}(x_{0}^{j})$, keeping the same notation $u_{j}$ for this extension, we have, $\widetilde{u}=\sum_{i\in \mathcal J^n }u_{j}$. Using magnetic translation invariance, it is easy to check that
$$F^{+1}_{b,Q_{nR}}(\widetilde{u})=\sum_{j\in\mathcal J^n } F^{+1}_{b,Q_{R}}(u_{j})=n^{2}e_{D}(b,R)\,.$$
Consequently, we get 
$$e_{D}(b,nR)\leq n^{2}e_{D}(b,R)\,.$$
We now divide both sides of this inequality by $n^{2}R^{2}$ then we take the limit as $n\longrightarrow\infty$. Having in mind \eqref{f(b)}, this gives \eqref{eD>f(b)}.

\textbf{We prove \eqref{eD<f(b)}}.\\
 If  $ n\in \mathbb N^*$ and $j=(j_{1},j_{2})\in\Z^{2}$, we denote by
$$K_{j}=I_{j_{1}}\times I_{j_{2}}\,,$$
where 
$$\forall m\in\Z,\qquad I_{m}=\left(\frac{2m+1-n}{2}-\frac{1}{2},\frac{2m+1-n}{2}+\frac{1}{2}\right)\,.$$
For all $R>0$, we set
$$Q_{R,j}=\{Rx:\,x\in K_{j}\}\,.$$
Let $\mathcal J^n =\{j=(j_{1},j_{2})\in\Z^{2}:0\leq j_{1},j_{2}\leq n-1\}$ and $Q_{nR}=\Big(-\frac{nR}{2},\frac{nR}{2}\Big)\times\Big(-\frac{nR}{2},\frac{nR}{2}\Big)$. Then the family  $(\overline{Q}_{R,j})$ is a covering of $Q_{nR}$, formed exactly of $n^{2}$ squares. Let $u=u_{nR}\in H^{1}_{0}(Q_{nR})$ be a minimizer of of $F^{+1}_{b,Q_{nR}}$ i.e. $F^{+1}_{b,Q_{nR}}(u)=e_D (b,nR)$. We have the obvious decomposition,
\begin{equation}\label{u1}
\int_{Q_{nR}}|u(x)|^{4}\,dx=\sum_{i\in\mathcal J^n }\int_{Q_{R,j}}|u(x)|^{4}\,dx\,.
\end{equation}
Let $\chi=\chi_{R,b^{\frac{1}{2}}}(x-x_{0}^{j})$, where $\chi_{R,b^{\frac{1}{2}}}$ is the cut-off function introduced in \eqref{def-chiR}. The function $u$ satisfies $-b(\nabla-i\Ab_{0})^{2}u=(1-|u|^{2})u$ in $Q_{nR}$. It results from an integration by parts that
 \begin{equation}\label{f4}
 e_D(b,nR)=F^{+1}_{b,Q_{nR}}(u)= - \frac 12 \int_{Q_{nR}} (|u(x)|^4-1)\, dx\,.
 \end{equation}
 We may write,
 \begin{align*}
 \int_{Q_{R,j}} |(\nabla-i\Ab_{0})\chi u|^{2}\,dx &=\Big\langle(\nabla-i\Ab_{0})\chi u,(\nabla-i\Ab_{0})\chi u\Big\rangle\\
 &=\Big\langle\nabla\chi\,u,\nabla\chi\,u\Big\rangle+\Big\langle\chi\,(\nabla-i\Ab_{0})u,\chi\,(\nabla-i\Ab_{0})u\Big\rangle\\
 &\qquad\qquad\qquad\qquad\qquad\quad+2\Big\langle\nabla\chi\,u,\chi\,(\nabla-i\Ab_{0})u\Big\rangle\\
 &=\Big\langle\nabla\chi\,u,\nabla\chi\,u\Big\rangle+\Big\langle(\nabla-i\Ab_{0})(\chi^{2}\,u),(\nabla-i\Ab_{0})u\Big\rangle.
 \end{align*}
 An integration by parts gives us 
 \begin{equation}\label{IP}
  \int_{Q_{R,j}} |(\nabla-i\Ab_{0})\chi u|^{2}\,dx=\int_{Q_{R,j}} |\nabla\chi|^2|u|^2\,dx-\Big\langle\chi^{2}\,u,(\nabla-i\Ab_{0})^{2}u\Big\rangle.
 \end{equation}
Using \eqref{IP}, we may express the energy $F^{+1}_{b,Q_{R,j}}(\chi u)$ as follows:
 \begin{align*}
F^{+1}_{b,Q_{R,j}} (\chi u) &= \int_{Q_{R,j}}\Big( b|(\nabla-i\Ab_{0})\chi u|^{2}-|\chi u|^{2}\Big)\,dx+\frac 12  \int_{Q_{R,j}} \Big( |\chi u|^4+1\Big) dx\\
&=-\langle \chi^2u, (b(\nabla-i\Ab_{0})^{2} +1) u \rangle + b \int_{Q_{R,j}} |\nabla \chi|^2 |u|^2\,dx + \frac 12  \int_{Q_{R,j}} (\chi^4 |u|^4+1) dx\,. 
 \end{align*}
 Using the equation $(b(\nabla-i\Ab_{0})^{2} +1) u = |u|^2 u$ and the inequality $\chi^{4}\leq\chi^{2}$, we get
 \begin{align*}
 F^{+1}_{b,Q_{R,j}} (\chi u) &\leq b \int_{Q_{R,j}} |\nabla \chi|^2 |u|^2\,dx - \frac 12  \int_{Q_{R,j}} \Big(\chi^2 |u|^4-1\Big) dx\\
 &\leq b \int_{Q_{R,j}} |\nabla \chi|^2 |u|^2\,dx+\frac 12  \int_{Q_{R,j}}\Big(1-\chi^{2}\Big) \,dx- \frac 12  \int_{Q_{R,j}} \Big(|u|^4-1\Big) dx\\
 &\leq - \frac 12 \int_{Q_{R,j}} \Big(|u|^4-1\Big)\,dx + C b^{\frac{1}{2}}R\,.
  \end{align*}
  Since each $\chi u$ has support in a square of side length $R$, we get
 \begin{equation}\label{f5}
 F^{+1}_{b,Q_{R,j}} (\chi u) \geq e_D (b,R)\,.
 \end{equation} 
  We sum over  the $n^{2}$  squares $(Q_{R,j})_{j\in\mathcal J^n }$ (that cover $Q_{nR}$), and get
  $$
  n^2 e_D (b,R) \leq - \frac 12 \int_{Q_{nR}} (|u|^4-1)\,dx  + C  b^{\frac{1}{2}}R n^2\,.
  $$
  Using \eqref{f4}, we obtain
  $$
  n^2 e_D (b,R) \leq e_D (b,nR) +  Cn^2 R b^\frac 12\,.
  $$
  Dividing by $n^2 R^2$, we obtain
  $$
  \frac{e_D (b,R)}{R^2} \leq \frac{e_D (b, nR)}{(nR)^2} + C R^{-1} b^\frac 12\,.
  $$
  We take the limit $n\rightarrow + \infty$ and get \eqref{eD<f(b)}.
\end{proof}
%%%%%%%%%%%%%%%%%%%%%%%%%%%%%%%%%%%%%%%%%%%%%%%%%%%%%%%%%%%%%%%%%%%%%%%%%%%%%%%%%%%%%%%%%%%%%%%%%%%%%%%%%%
\section{Upper bound of the energy}\label{upperbound}
The aim of this section is to give an upper bound on the ground state energy $\E0(\kappa,H)$ introduced in \eqref{eq-2D-gs}.\\
In the sequel, for some choice of $\rho\in(0,1)$ to be determined later  (see \eqref{choice-ell-rho}), we consider triples $(\ell,x_0,\widetilde{x}_{0})$ such that $\overline{Q_{\ell}(x_0)}\subset\{|B_{0}|>\rho\}\cap\Omega$ and $\widetilde{x}_{0}\in\overline{Q_{\ell}(x_{0})}$. In this situation, we say that this triple is $\rho$-admissible, that the pair $(\ell,x_0)$ is $\rho$-admissible and the corresponding square $Q_\ell (x_0)$ is a $\rho$-admissiblle.
 Let us introduce the function:
\begin{equation}\label{def-w}
w_{\ell,x_0,\widetilde{x}_{0}}(x)=\begin{cases}
e^{i\kappa H\varphi_{x_{0},\widetilde{x}_{0}}}u_{R}\left(\frac{R}{\ell}(x-x_{0})\right)&{\rm if}~x\in Q_{\ell}(x_0)\subset\{B_{0}>\rho\}\cap\Omega\\
e^{i\kappa H\varphi_{x_{0},\widetilde{x}_{0}}}\overline{u}_{R}\left(\frac{R}{\ell}(x-x_{0})\right)&{\rm if}~x\in Q_{\ell}(x_0)\subset\{B_{0}<-\rho\}\cap\Omega\,,
\end{cases}
\end{equation}
where $u_{R}\in H^{1}_{0}(\Omega)$ is a minimizer of the functional in \eqref{eq-GL-F} and $\varphi_{x_{0},\widetilde{x}_{0}}$ is the function introduced in \cite[Lemma~A.3]{KA} that satisfies
\begin{equation}\label{F-A}
|\Fb(x)-\sigma_{\ell}|B_{0}(\widetilde{x}_{0})|\Ab_{0}(x-x_{0})-\nabla\varphi_{x_{0},\widetilde{x}_{0}}(x)|\leq C\ell^{2},\,\qquad\,\,\left( x \in Q_{\ell}(x_{0})\right)\,,
\end{equation}
where $B_{0}(\widetilde{x}_{0})=\curl\Fb(\widetilde{x}_{0})$, $\Ab_{0}$ is the magnetic potential introduced in \eqref{eq-hc2-mpA0} and $\sigma_{\ell}$ is the sign of $B_{0}(x)$ in $Q_{\ell}(x_{0})$.
%%%%%%%%%%%%%%%%%%%%%%%%%%%%%%%%%%%%%%%%%%PPPPPPPPPPPPP%%%%%%%%%%%%%%%%%%%%%%%%%%%%%%%%%%%%%%%%%%%%%%%%%%%%
\begin{prop}\label{pp-up-Eg}
Under Assumption~\eqref{B(x)}, there exist positive constants $C$ and $\kappa_{0}$ such that if $\kappa\geq\kappa_{0}$, $\ell\in(0,1)$, $\delta\in(0,1)$, $\rho>0$, $\ell^{2}\kappa H\rho>1$, and $(\ell,x_0,\widetilde x_0)$ is a $\rho$-admissible triple, then,
\begin{equation}\label{up-Eg-eq}
\frac{1}{|Q_{\ell}(x_0)|}\mathcal{E}_{0}(w_{\ell,x_0,\widetilde{x}_{0}},\Fb,Q_{\ell}(x_0))\leq (1+\delta)\kappa^{2} \hat{f}\left(\frac{H}{\kappa}|B_{0}(\widetilde{x}_{0})|\right)+C\left(\frac{1}{\ell H}+\delta^{-1}\ell^{4}\kappa H\right)\kappa H\,.
\end{equation}
\end{prop}
\begin{proof}
Let 
\begin{equation}\label{def-Rb}
R=\ell\sqrt{\kappa H |B_{0}(\widetilde{x}_{0})|}\qquad{\rm and}\qquad b=\frac{H}{\kappa}|B_{0}(\widetilde{x}_{0})|\,.
\end{equation}
We estimate $\mathcal E_{0}(w_{\ell,x_0,\widetilde{x}_{0}},\Fb,Q_{\ell}(x_0))$ from above. We write for any $\delta\in(0,1)$
\begin{align}
&\mathcal E_{0}(w_{\ell,x_0,\widetilde{x}_{0}},\Fb ,Q_{\ell}(x_0))\nonumber\\
&\qquad\qquad=\int_{Q_{\ell}(x_0)}\left[|(\nabla-i\kappa H \Fb)w_{\ell,x_0,\widetilde{x}_{0}}|^{2}+\frac{\kappa^2}{2}(1-|w_{\ell,x_0,\widetilde{x}_{0}}|^2)^2\right]\,dx\nonumber\\
&\qquad\qquad\leq(1+\delta)\int_{Q_{\ell}(x_0)}\left[|(\nabla-i\kappa H (\sigma_{\ell}|B_{0}(\widetilde{x}_{0})|\Ab_{0}(x-x_0)+\nabla\varphi_{x_{0},\widetilde{x}_{0}}(x)))w_{\ell,x_0,\widetilde{x}_{0}}|^{2}\right.\nonumber\\
&\qquad\qquad\qquad\qquad\qquad\qquad\qquad\qquad\qquad\qquad\qquad\qquad\qquad\qquad\left.+\frac{\kappa^2}{2}(1-|w_{\ell,x_0,\widetilde{x}_{0}}|^2)^2\right]dx\nonumber\\
&\qquad\qquad\qquad+C(\kappa H)^{2}\delta^{-1}\int_{Q_{\ell}(x_0)}\left|\Fb- (\sigma_{\ell}|B_{0}(\widetilde{x}_{0})|\Ab_{0}(x-x_0)+\nabla\varphi_{x_{0},\widetilde{x}_{0}}(x))w_{\ell,x_0,\widetilde{x}_{0}}\right|^{2}\,dx\nonumber\\
&\qquad\qquad\leq(1+\delta)\mathcal E_{0}\big(w_{\ell,x_0,\widetilde{x}_{0}},\,\sigma_{\ell}|B_{0}(\widetilde{x}_{0})|\Ab_{0}(x-x_0)+\nabla\varphi_{x_{0},\widetilde{x}_{0}}(x),\,Q_{\ell}(x_0)\big)+C\delta^{-1}\ell^{6}(\kappa H)^{2}\label{1up-loc-en}.
\end{align}
Using \eqref{nE=pE}, the definition of $w_{\ell,x_0,\widetilde{x}_{0}}$ and the change of variable $y=\frac{R}{\ell}(x-x_0)$, we obtain
\begin{align}\label{2up-loc-en}
&\mathcal E_{0}\big(w_{\ell,x_0,\widetilde{x}_{0}},\,\sigma_{\ell}|B_{0}(\widetilde{x}_{0})|\Ab_{0}(x-x_0)+\nabla\varphi_{x_{0},\widetilde{x}_{0}}(x),\,Q_{\ell}(x_0)\big)\nonumber\\
&=\int_{Q_{R}}\left[\left|\left(\frac{R}{\ell}\nabla_{y}-i\frac{R}{\ell}\Ab_{0}(y)\right)u_{R}(y)\right|^{2}+\frac{\kappa^2}{2}\left(1-\left|u_{R}(y)\right|^2\right)^2\right]\frac{\ell^2}{R^2} \,dy\nonumber\\
&=\frac{1}{b} F^{+1}_{b,Q_{R}}(u_{R})\,.
\end{align}
Since $u_{R}\in H^{1}_{0}(Q_{R})$ is a minimizer of $ F^{+1}_{b,Q_{R}}$, then
\begin{equation}\label{F=eD}
F^{+1}_{b,Q_{R}}(u_{R})=e_{D}(b,R)\,.
\end{equation}
Proposition~\ref{FR} tells us that $\,\,\,\displaystyle\frac{e_{D}(b,R)}{R^{2}}\leq \hat{f}(b)+C\frac{\sqrt{b}}{R}$ for all $b\in]0,1[$ and $R\geq 1$.
This assumption is satisfied because $R\geq \ell\sqrt{\kappa H \rho}>1$ (see Remark~\ref{choice-ell1}). Therefore, we get from \eqref{2up-loc-en} and \eqref{F=eD} the estimate
\begin{equation}\label{3up-loc-en}
\mathcal E_{0}(w_{\ell,x_0,\widetilde{x}_{0}},\sigma_{\ell}|B_{0}(\widetilde{x}_{0})|\Ab_{0}(x-x_0)+\nabla\varphi_{x_{0},\widetilde{x}_{0}}(x),Q_{\ell}(x_0))\leq R^{2}\frac{\hat{f}(b)}{b}+C\frac{R}{\sqrt{b}}\,,
\end{equation}
with $b$ defined in \eqref{def-Rb}.\\
We get by collecting the estimates in \eqref{1up-loc-en}-\eqref{2up-loc-en} that,
\begin{equation}\label{up-E0}
\mathcal{E}_{0}(w_{\ell,x_0,\widetilde{x}_{0}},\Fb,Q_{\ell}(x_0))\leq (1+\delta) R^{2}\frac{\hat{f}(b)}{b}+C_{1}\frac{R}{\sqrt{b}}+C_{2}\delta^{-1}\ell^{6}(\kappa H)^{2}\,.
\end{equation}
Remembering the definition of $b$ and $R$ in \eqref{def-Rb}, we get
$$
\frac{1}{|Q_{\ell}(x_0)|}\mathcal{E}_{0}(w_{\ell,x_0,\widetilde{x}_{0}},\Fb,Q_{\ell}(x_0))\leq (1+\delta)\kappa^{2} \hat{f}\left(\frac{H}{\kappa}|B_{0}(\widetilde{x}_{0})|\right)+C\left(\frac{\kappa}{\ell}+\delta^{-1}\ell^{4}(\kappa H)^{2}\right)\,,
$$
which finishes the proof of Proposition~\ref{pp-up-Eg}.
\end{proof}
%%%%%%%%%%%%%%%%%%%%%%%%%%%%%%%%%%%%RRRRRRRRRRRRRRRR%%%%%%%%%%%%%%%%%%%%%%%%%%%%%%%%%%%%%%%%%%%%%%%%%%%%
\begin{rem}\label{choice-ell1}
We select $\ell,\delta$ and $\rho$ as follow:
\begin{equation}\label{choice-ell-rho}
\ell=\left(\kappa H\right)^{-\frac{1}{4}}\,,\qquad \rho=\left(\kappa H\right)^{-\frac{1}{3}}\,.
\end{equation}
and
\begin{equation}\label{choice-delta}
\delta=\left(\ln\frac{\kappa}{H}\right)^{-\frac{1}{4}}
\end{equation}
Under Assumption~\eqref{cond-H}, this choice permits us to verify the assumptions in Proposition~\ref{pp-up-Eg} and to obtain error terms of order $\textit{o}\left(\kappa H \ln\frac{\kappa}{H}\right)$. We have indeed  as $\kappa\longrightarrow+\infty$
$$\frac{\kappa}{\ell\kappa H\ln\frac{\kappa}{H}}=\frac{\kappa^{\frac{1}{4}}}{  H^\frac{3}{4}\ln\frac{\kappa}{H}}\ll 1\,,$$
$$\frac{\delta^{-1}(\kappa H)^{2}\ell^{4}}{\kappa H \ln\frac{\kappa}{H}}=\frac{1}{\left(\ln\frac{\kappa}{H}\right)^{\frac{3}{4}}}\ll 1\,,$$
$$\ell(\kappa H)^\frac{1}{2} \rho^\frac{1}{2}=\left(\kappa H\right)^{\frac{1}{12}}\gg1\,.$$
\end{rem}
%%%%%%%%%%%%%%%%%%%%%%%%%%%%%%%%%%%%TTTTTTTTTTTTTTT%%%%%%%%%%%%%%%%%%%%%%%%%%%%%%%%%%%%%%%%%%%%%%%%%%%
\begin{theorem}\label{up-Eg}
Under Assumption~\eqref{B(x)}, if \eqref{cond-H} holds, then, the ground state energy $\E0(\kappa,H)$ in \eqref{eq-2D-gs} satisfies
\begin{equation}\label{up-Eg-eq1}
\E0(\kappa,H)\leq \kappa^{2}\int_{\Omega}  \hat{f}\left(\frac{H}{\kappa}|B_{0}(x)|\right)\,dx+\textit{o}\left(\kappa H\ln\frac{\kappa}{H}\right)\,,\quad\text{as}~\kappa\longrightarrow+\infty\,.
\end{equation}
\end{theorem}
\begin{proof}
Let $\ell\in(0,1)$, $\delta$ and $\rho$ be the parameters depending on $\kappa$ and chosen as in Remark~\ref{choice-ell1}.\\
As we did in the previous paper \cite[Proposition~5.1]{KA}, we consider the lattice $\Gamma_{\ell}:=\ell\Z\times\ell\Z$ and write, for $\gamma, \widetilde{\gamma}\in\Gamma_{\ell}\,$,
$$
 Q_{\gamma,\ell}=Q_{\ell}(\gamma)\mbox{ and }w_{\ell,x_0,\widetilde{x}_{0}}=w_{\ell,\gamma,\widetilde{\gamma}}\,.$$
 For any $\gamma\in \Gamma_{\ell}$ such that $Q_{\gamma,\ell}$ is $\rho$-admissible square, let
\begin{equation}\label{inf-B}
\underline{B}_{\gamma,\ell}=\inf_{x\in Q_{\gamma,\ell}}|B_{0}(x)|
\end{equation}
and
\begin{equation}\label{def-card-squares}
\mathcal I_{\ell, \rho}=\left\{\gamma; \,\overline{Q_{\gamma, \ell}}\subset \Omega\cap\left\{|B_{0}|>\rho\right\}\right\}\,,\qquad N={\rm card}~ \mathcal I_{\ell, \rho}\,.
\end{equation}
Then as $\kappa\rightarrow+\infty$, we have:
\begin{equation}
N=|\Omega|\ell^{-2}+\mathcal{O}(\ell^{-1})+\mathcal O(\rho\ell^{-2}).
\end{equation}
For all $x\in\Omega$, we define,
\begin{equation}
s(x)=
\sum_{\gamma\in\mathcal{J}_{\ell,\rho}} \,w_{\ell,\gamma,\widetilde{\gamma}}(x)\,,
\end{equation}
where $w_{\ell,\gamma,\widetilde{\gamma}}$ has been extended by $0$ outside of $Q_{\gamma,\ell}$. Remember the functional $\mathcal{E}_{\kappa, H}$ in \eqref{eq-2D-GLf}. We compute the energy of the test configuration $(s,\Fb)$. Since $\curl \Fb=B$, we get,
\begin{align}\label{sumE1}
\mathcal{E}_{\kappa,H}(s,\Fb,\Omega)=\sum_{\gamma\in\mathcal{J}_{\ell,\rho}}\mathcal{E}_{0}(w_{\ell,\gamma,\widetilde{\gamma}},\Fb,Q_{\gamma, \ell})\,.
\end{align}
Recall that for any $\widetilde{\gamma}\in Q_{\gamma,\ell}$, $B_{0}(\widetilde{\gamma})$ satisfies \eqref{F-A}. Then, we select $\widetilde{\gamma}\in Q_{\gamma,\ell}$ such that 
$$|B_{0}(\widetilde{\gamma})|=\underline{B}_{\gamma,\ell}\,.$$ Using Proposition~\ref{pp-up-Eg} and noticing that $|Q_{\gamma,\ell}|=\ell^2$, we get for any $\delta\in(0,1)$
\begin{equation}\label{sumE2}
\sum_{\gamma\in\mathcal{J}_{\ell,\rho}}\mathcal{E}_{0}(w_{\ell,\gamma,\widetilde{\gamma}},\Fb,Q_{\gamma, \ell})\leq \kappa^{2}(1+\delta)\sum_{\gamma\in\mathcal{J}_{\ell,\rho}} \hat{f}\left(\frac{H}{\kappa}\underline{B}_{\gamma,\ell}\right)\ell^{2}+r(\kappa,H,\ell)\,,
\end{equation}
where 
\begin{equation}\label{remainder}
r(\kappa,H,\ell)=\mathcal{O}\left(\frac{\kappa}{\ell}+\delta^{-1}\ell^{4}(\kappa H)^{2}\right)\,.
\end{equation}
Having in mind Property~(3) of the function $\hat{f}$ established in Theorem~\ref{pro-f(b)}, we recognize the lower Riemann sum and notice that $\cup_{\gamma\in \mathcal{J}_{\ell, \rho}}Q_{\gamma, \ell}\subset\Omega$, then, we get by collecting \eqref{sumE1}-\eqref{sumE2} that
\begin{equation}\label{final-est}
\mathcal{E}_{\kappa,H}(s,\Fb,\Omega)\leq(1+\delta)\kappa^{2}\int_{\Omega} \hat{f}\left(\frac{H}{\kappa}|B_{0}(x)|\right)\,dx+r(\kappa,H,\ell)\,.
\end{equation}
The choice of the parameters $\delta$ in \eqref{choice-delta} and $\ell$ in \eqref{choice-ell-rho} implies that all error terms are of lower order compared to $\kappa H\ln\frac{\kappa}{H}$.\\
\end{proof}

\begin{rem}\label{R1}
The remainder term in \eqref{remainder} is small compared with the leading order term. We have, for any $\rho_{0}>0$
\begin{align*}
\kappa^{2}\int_{\Omega}  \hat{f}\left(\frac{H}{\kappa}|B_{0}(x)|\right)\,dx&\geq \kappa^{2}\int_{\Omega\cap\{|B_{0}|>\rho_{0}\}}  \hat{f}\left(\frac{H}{\kappa}|B_{0}(x)|\right)\,dx\\
&\geq \kappa^{2}\hat{f}\left(\frac{H}{\kappa}\rho_{0}\right)|\Omega\cap\{|B_{0}|>\rho_{0}\}|\,.
\end{align*}
In view of \eqref{main-f(b)}, for all positive constant $C$ there exists $\rho_{0}>0$ such that if $H\leq C\kappa$ and $\frac{\rho_{0}}{C_{1}}$ is sufficiently small for some $C_{1}>0$, then as $\kappa\longrightarrow+\infty$
$$\kappa^{2}\int_{\Omega}  \hat{f}\left(\frac{H}{\kappa}|B_{0}(x)|\right)\,dx\geq C_{2}\frac{\kappa H\rho_{0}}{2}\ln\frac{\kappa}{H\rho_{0}}(1+\textit{o}(1))\,,$$
where $C_{2}$ is a positive constant.\\
In particular, when \eqref{cond-H} is satisfied, we see that,
\begin{equation}
r(\kappa, H,\ell)\ll\kappa^{2}\int_{\Omega}  \hat{f}\left(\frac{H}{\kappa}|B_{0}(x)|\right)\,dx\,.
\end{equation}
\end{rem}
%%%%%%%%%%%%%%%%%%%%%%%%%%%%%%%%%%%%%%%%%%%%%%%%%%%%%%%%%%%%%%%%%%%%%%%%%%%%%%%%%%%%%%%%%%%%%%%%%%%%%%%
\section{A priori estimates of minimizers}
The aim of this section is to give a priori estimates on the solutions of the Ginzburg-Landau equations \eqref{eq-2D-GLeq}. These estimates play an essential role in controlling the error resulting from various approximations.
The starting point is the following $L^{\infty}$-bound resulting from the maximum principle. If $(\psi, \Ab)\in H^{1}(\Omega;\C)\times H^{1}_{\rm div}(\R^{2})$ is a solution of \eqref{eq-2D-GLeq}, then 
\begin{equation}\label{est-psi}
\|\psi\|_{L^{\infty}(\Omega)}\leq 1\,.
\end{equation}
Next we prove an estimate on the induced magnetic potential.
%%%%%%%%%%%%%%%%%%%%%%%%%%%%%%%%%%%%%%%%PPPPPPPPPPPPPPPPPP%%%%%%%%%%%%%%%%%%%%%%%%%%%%%%%%%%%%%%%%%%%%%%%%%%
\begin{prop}\label{pr-est}
Suppose that the magnetic field $H$ is a function of $\kappa$ and satisfies \eqref{cond-H}. Let $\alpha\in(0,1)$. There exist positive constants $\kappa_{0}$ and $C$ such that, if $\kappa\geq\kappa_{0}$ and $(\psi, \Ab)$ is a minimizer of \eqref{eq-2D-GLf}, then
$$\|\Ab-\Fb\|_{H^{2}(\Omega)}\leq \frac{C}{H}\left(\int_{\Omega}  \hat{f}\left(\frac{H}{\kappa}|B_{0}(x)|\right)\,dx\right)^{\frac{1}{2}}\,,$$
$$\|\Ab-\Fb\|_{C^{0,\alpha}(\overline{\Omega})}\leq  \frac{C}{H}\left(\int_{\Omega}  \hat{f}\left(\frac{H}{\kappa}|B_{0}(x)|\right)\,dx\right)^{\frac{1}{2}}\,.$$
Here $\Fb$ is the magnetic potential introduced in \eqref{div-curlF}.
\end{prop}
\begin{proof}
The estimate in $C^{0,\alpha}$-norm is a consequence of the continuous Sobolev embedding of $H^{2}(\Omega)$ in $C^{0,\alpha}(\overline{\Omega})$.\\
It is easy to show that
\begin{equation}\label{est-curla1}
\|\curl(\Ab-\Fb)\|_{L^{2}(\Omega)}\leq\frac{1}{\kappa H}\E0(\kappa,H)^{\frac{1}{2}}\,,
\end{equation}
and
\begin{equation}\label{est-grad1}
\|(\nabla-i\kappa H\Ab)\psi\|_{L^{2}(\Omega)}\leq\E0(\kappa,H)^{\frac{1}{2}}\,.
\end{equation}
Notice that under Assumption~\eqref{cond-H}, it follows from Theorem~\ref{up-Eg} and Remark~\ref{R1} that
\begin{equation}\label{est-curla}
\|\curl(\Ab-\Fb)\|_{L^{2}(\Omega)}\leq\frac{1}{H}\left(\int_{\Omega}  \hat{f}\left(\frac{H}{\kappa}|B_{0}(x)|\right)\,dx\right)^{\frac{1}{2}}\,,
\end{equation}

\begin{equation}\label{est-grad}
\|(\nabla-i\kappa H\Ab)\psi\|_{L^{2}(\Omega)}\leq \kappa\left(\int_{\Omega}  \hat{f}\left(\frac{H}{\kappa}|B_{0}(x)|\right)\,dx\right)^{\frac{1}{2}}\,.
\end{equation}
Let $a=\Ab-\Fb$. We will prove that $$\|a\|_{H^{2}(\Omega)}\leq\frac{C}{H}\left(\int_{\Omega}  \hat{f}\left(\frac{H}{\kappa}|B_{0}(x)|\right)\,dx\right)^{\frac{1}{2}}\,.$$ Since ${\rm div} a=0$ and $a\cdot\nu=0$ on $\partial \Omega$, we get by regularity of the curl-div system see \cite[ Appendix~A.5]{SB}
\begin{equation}\label{est-curl-div}
\|a\|_{H^{2}(\Omega)}\leq C' \|\curl a\|_{H^{1}(\Omega)}\,.
\end{equation}
The second equation in \eqref{eq-2D-GLeq} reads as follows:
$$-\nabla^{\perp} {\rm curl}\,a=\frac{1}{\kappa H}{\rm Im}(\overline{\psi}(\nabla-i\kappa H\Ab)\psi)\,.$$
The estimates in \eqref{est-psi} and the bound in \eqref{est-curl-div}, give us
$$
\|a\|_{H^{2}(\Omega)}\leq C\left(\|\curl a\|_{L^{2}(\Omega)}+\frac{1}{\kappa H}\|(\nabla-i\kappa H\Ab)\psi\|_{L^{2}(\Omega)}\right)\,.
$$
Inserting the estimates in \eqref{est-curla} and \eqref{est-grad} into this upper bound finishes the proof of the proposition.
\end{proof}
%%%%%%%%%%%%%%%%%%%%%%%%%%%%%%%%%%%%%%%%%%%%%%%%%%%%%%%%%%%%%%%%%%%%%%%%%%%%%%%%%%%%%%%%%%%%%%%%%%%%%%%%%
\section{Proof of Theorem~\ref{th-loc-enrgy}: Lower bound}\label{section5}
In this section, we suppose that $\mathcal{D}$ is an open set with smooth boundary such that $\overline{\mathcal{D}}\subset\Omega$. We will give a lower bound of the energy $\mathcal{E} (\psi,\Ab;\mathcal{D})$ introduced in \eqref{eq-GLen-D}, when $(\psi,\Ab)$ is a minimizer of the functional in \eqref{eq-2D-GLf}.

\subsection*{Construction of a gauge transformation:}~\\
Let $\phi_{x_{0}}(x)=(\Ab(x_0)-\Fb(x_0))\cdot x$, where $\Fb$ is the magnetic potential introduced in \eqref{div-curlF} and $(\ell,x_0)$
 a $\rho$-admissible pair. Choosing  $\alpha\in (0,1)$ and using the estimate of $\|\Ab-\Fb\|_{C^{0,\alpha}(\Omega)}$ given in Proposition~\ref{pr-est}, we get for all $x\in Q_{\ell}(x_0)$,
\begin{align}
|\Ab(x)-\nabla\phi_{x_0}-\Fb(x)|&=|(\Ab-\Fb)(x)-(\Ab-\Fb)(x_0)|\nonumber\\
&\leq\|\Ab-\Fb\|_{C^{0,\alpha}(\Omega)}|x-x_0|^{\alpha}\nonumber\\
& \leq C_\alpha \, \lambda\, \ell^{\alpha}\,,\label{alpha}
\end{align}
where 
\begin{equation}\label{lamda}
\lambda=\frac{1}{H}\left(\int_{\Omega}  \hat{f}\left(\frac{H}{\kappa}|B_{0}(x)|\right)\,dx\right)^{\frac{1}{2}}\,.
\end{equation}
Using \eqref{main-f(b)}, it is clear that, under condition \eqref{cond-H} 
\begin{equation}\label{b-lambda}
\lambda^{2}=\mathcal{O}\left(\frac{1}{\kappa H}\ln\frac{\kappa}{H}\right)\,,
\end{equation}
as $\kappa\longrightarrow+\infty$.
%%%%%%%%%%%%%%%%%%%%%%%%%%%%%%%%%%%%%%%%PPPPPPPPPPPPPPPPP%%%%%%%%%%%%%%%%%%%%%%%%%%%%%%%%%%%%%%%%%%%%%%%%%%%%

\begin{prop}\label{prop-lb}
For all $\alpha\in(0,1)$, there exist positive
constants $C$ and $\kappa_{0}$ such that if $\kappa\geq\kappa_{0}$, $\ell \in (0,\frac 12)$, $\delta\in(0,1)$, $\rho>0$, $\ell^{2}\kappa H \rho>1$, $(\psi,\Ab)\in H^1(\Omega;\C)\times H^1_{\Div}(\Omega)$ is a minimizer of \eqref{eq-2D-GLf}, and $(\ell,x_0,\widetilde x_0)$ a $\rho$-admissible triple, then,
$$\frac1{|Q_{\ell}(x_0)|}\mathcal E_0(\psi,\Ab;Q_{\ell}(x_0))\geq (1-\delta)\kappa^2 \hat{f}\left(\displaystyle\frac{H}{\kappa}|B_{0}(\widetilde{x}_{0})|\right)-C\left( \frac{\kappa}{\ell}+\delta^{-1}(\kappa H)^{2}\ell^{4}+\delta^{-1}\kappa^{2}H^{2}\lambda^{2}\ell^{2\alpha} \right)\,.$$
\end{prop}
\begin{proof}
Let $\widetilde{x}_{0}\in\overline{ Q_{\ell}(x_{0})}$. Recall the function $\varphi_{x_{0},\widetilde{x}_{0}}$ satisfiying \eqref{F-A}. For all $x\in Q_{\ell}(x_{0})$, let
\begin{equation}\label{defu}
u(x)=e^{-i\kappa H\varphi}\psi(x)\,,
\end{equation}
where $\varphi=\varphi_{x_{0},\widetilde{x}_{0}}+\phi_{x_{0}}$ and $\phi_{x_{0}}$ is introduced in \eqref{alpha}.

\textbf{Estimate of $\mathcal{E}_0$ in $Q_{\ell}(x_0)$:}~\\
As we did in \cite[Lemma~4.1]{KA}, we have, for any
$\delta\in(0,1)$ and $\alpha\in(0,1)$
\begin{align}\label{es of A}
\mathcal E_0(\psi,\Ab;Q_{\ell}(x_0))&\geq (1-\delta)\mathcal
E_0(u,\sigma_{\ell}|B_{0}(\widetilde{x}_{0})|\Ab_{0}(x-x_{0});Q_{\ell}(x_0))\nonumber\\
&\qquad\qquad\qquad\qquad\qquad\qquad-C\delta^{-1}(\kappa H)^{2}\left(\ell^{4}+\lambda^{2}\ell^{2\alpha}\right)\int_{Q_{\ell}(x_0)}|\psi|^2\,dx.
\end{align}
Let
\begin{equation}\label{def-bR}
R=\ell\sqrt{\kappa H |B_{0}(\widetilde{x}_{0})|}\qquad{\rm and}\qquad b=\frac{H}{\kappa}|B_{0}(\widetilde{x}_{0})|\,.
\end{equation}
 Define the function in $Q_{R}$
\begin{equation}\label{def-v}
v_{\ell,x_0,\widetilde{x}_{0}}(x)=\begin{cases}
u\left(\frac{\ell}{R} x+x_{0}\right)&{\rm if}~x\in Q_{R}\subset\{\{B_{0}>\rho\}\cap\Omega\}\\
\overline{u}\left(\frac{\ell}{R} x+x_{0}\right)&{\rm if}~x\in Q_{R}\subset\{\{B_{0}<-\rho\}\cap\Omega\}\,.
\end{cases}
\end{equation}
Using \eqref{nE=pE}, and the change of variable $y=\frac{R}{\ell}(x-x_{0})$, we get
\begin{equation}\label{E0<}
\mathcal E_0(u,\sigma_{\ell}\,|B_{0}(\widetilde{x}_{0})|\Ab_{0}(x-x_0);Q_{\ell}(x_0))=\frac{1}{b} F^{+1}_{b,Q_{R}}(v_{\ell,x_{0},\widetilde{x}_{0}})\,.
\end{equation}
Here $F^{+1}_{b,Q_{R}}$ is introduced in \eqref{eq-GL-F}. 
Since $v_{\ell,x_{0},\widetilde{x}_{0}}\in H^{1}(Q_{R})$, we have 
\begin{equation}\label{F>}
F^{+1}_{b,Q_{R}}(v_{\ell,x_{0},\widetilde{x}_{0}})\geq e_{N}(b,R)\,.
\end{equation}
By collecting \eqref{eD<ep}-\eqref{eD<f(b)} and the lower bound in \eqref{F>}, we get,
\begin{equation}\label{F2}
F^{+1}_{b,Q_{R}}(v_{\ell,x_{0},\widetilde{x}_{0}})\geq R^{2}\hat{f}(b)-CRb^{\frac{1}{2}}\,.
\end{equation}
As a consequence,  \eqref{E0<} gives us
\begin{equation}\label{final-maj}
\mathcal E_0(u,\sigma_{\ell}|B_{0}(\widetilde{x}_{0})|\Ab_{0}(x-x_0);Q_{\ell}(x_0))\geq R^{2}\frac{\hat{f}(b)}{b}-C\frac{R}{\sqrt{b}}\,.
\end{equation}
with $b$ and $R$ introduced in \eqref{def-bR}.\\
Inserting \eqref{final-maj} into \eqref{es of A} and using the bound of $\psi$ in \eqref{est-psi}, we get
$$\mathcal E_0(\psi,\Ab;Q_{\ell}(x_0))\geq (1-\delta)R^{2}\frac{\hat{f}(b)}{b}-C\frac{R}{\sqrt{b}}-C'\delta^{-1}(\kappa H)^{2}\left(\ell^{4}+\lambda^{2}\ell^{2\alpha}\right)\ell^{2}\,.$$
Having in mind \eqref{def-bR}, we get for any $\alpha\in(0,1)$
\begin{equation}\label{est-loc-E0}
\mathcal E_0(\psi,\Ab;Q_{\ell}(x_0))\geq \left( (1-\delta)\kappa^{2}\hat{f}\left( \frac{H}{\kappa}|B_{0}(\widetilde{x}_{0})| \right)-C\left( \frac{\kappa}{\ell}+\delta^{-1}(\kappa H)^{2}\ell^{4}+\delta^{-1}\kappa^{2}H^{2}\lambda^{2}\ell^{2\alpha} \right)\right)\ell^{2}\,.
\end{equation}
This finishes the proof of Proposition~\ref{prop-lb}.
\end{proof}
%%%%%%%%%%%%%%%%%%%%%%%%%%%%%%%%%%%%%%%%%RRRRRRRRRRRRRRRRRR%%%%%%%%%%%%%%%%%%%%%%%%%%%%%%%%%%%%%%%%%%%%%%
%\begin{rem}
%Under Assumptions~\eqref{cond-H}, use Proporty~\eqref{p6} of $\hat{f}$ in Theorem~\ref{pro-f(b)}, $CRb^{\frac{1}{2}}$ in \eqref{F2} is replaced by $R^{2}{\rm err}(b,R)$ where the function ${\rm err}$ is introduced in \eqref{err}.
%\end{rem}
%%%%%%%%%%%%%%%%%%%%%%%%%%%%%%%%%%%%%%%%RRRRRRRRRRRRRRRRRR%%%%%%%%%%%%%%%%%%%%%%%%%%%%%%%%%%%%%%%%%%%%%%
\begin{rem}\label{choice-ell2}
For any $\alpha\in(0,1)$, we keep the same choice of $\ell$, $\rho$ as in \eqref{choice-ell-rho} and choose $\delta$ as follows:
\begin{equation}\label{choice-delta-alpha}
\delta=\left(\ln\frac{\kappa}{H}\right)^{- \frac{\alpha}{4}}\,.
\end{equation}
This choice and Assumption~\eqref{cond-H}  permit us to have the assumptions of Proposition~\ref{prop-lb} satisfied and make the error terms in its statement
 of order $\textit{o}\left(\kappa H \ln\frac{\kappa}{H}\right)$.
We have as $\kappa\longrightarrow+\infty\,$,
$$\frac{\kappa}{\ell\kappa H\ln\frac{\kappa}{H}}=\frac{\kappa^{\frac{1}{4}}}{H^{\frac{3}{4}}\ln\frac{\kappa}{H}}\ll 1\,,$$
$$\frac{\delta^{-1}(\kappa H)^{2}\ell^{4}}{\kappa H \ln\frac{\kappa}{H}}=\frac{1}{\left(\ln\frac{\kappa}{H}\right)^{1-\frac{\alpha}{4}}}\ll 1\,,$$
$$\frac{\delta^{-1}\kappa H\ln\frac{C_{0}\kappa}{H}\ell^{2\alpha}}{\kappa H \ln\frac{\kappa}{H}}=\frac{\ln\frac{C_{0}\kappa}{H}}{\left(\ln\frac{\kappa}{H}\right)^{1-\frac{\alpha}{4}}\left(\kappa H\right)^{\frac{\alpha}{2}}}\ll 1\,,$$
$$\ell(\kappa H)^\frac{1}{2} \rho^\frac{1}{2}=\left(\kappa H\right)^{\frac{1}{12}}\gg1\,.$$
\end{rem}
%%%%%%%%%%%%%%%%%%%%%%%%%%%%%%%%%%%%%%%%%%%%%%RRRRRRRRRRRRRRRRR%%%%%%%%%%%%%%%%%%%%%%%%%%%%%%%%%%%%%
\begin{rem}\label{pr-e(F)<e(A)}
As a byproduct of the proof, we get also a useful estimate. Using the bound $|\psi|\leq 1$, it results from \eqref{es of A}:
\begin{multline}\label{eq-e(F)<e(A)}
\frac{(1-\delta)}{|Q_{\ell}(x_0)|}\mathcal{E}_{0}(\psi,\sigma_{\ell}|B_{0}(\widetilde{x}_{0})|\Ab_{0}(x-x_0)+\nabla\varphi,Q_{\ell}(x_0))\leq\frac{1}{|Q_{\ell}(x_0)|}\mathcal{E}_{0}(\psi,\Ab,Q_{\ell}(x_0))\\
+C\delta^{-1}(\kappa H)^{2}(\ell^{4}+\lambda^{2}\ell^{2\alpha})\,.
\end{multline}
Using \eqref{b-lambda} and choosing $\ell$, $\rho$ as in \eqref{choice-ell-rho} and $\delta$ as in \eqref{choice-delta-alpha}, we get a function $\hat{\textit{r}}:(0,+\infty)\longmapsto(0,+\infty)$ satisfying $\displaystyle \lim_{t\longrightarrow+\infty}\hat{\textit{r}}(t)=0$ and
\begin{equation}\label{eq-e(F)<e(A)2}
\mathcal{E}_{0}(\psi,\sigma_{\ell}|B_{0}(\widetilde{x}_{0})|\Ab_{0}(x-x_0)+\nabla\varphi,Q_{\ell}(x_0))\leq\mathcal{E}_{0}(\psi,\Ab,Q_{\ell}(x_0))+\ell^{2}\kappa H \ln\frac{\kappa}{H}\,\hat{\textit{r}}(\kappa)\,,
\end{equation}
for any  $\widetilde{x}_{0}$ in $Q_{\ell}(x_{0})\,$.
\end{rem}

The next theorem presents the lower bound of the local energy in the domain $\mathcal{D}$ such that $\overline{\mathcal{D}}\subset\Omega$ and we deduce the lower bound of the global energy by replacing $\mathcal{D}$ with $\Omega$.
%%%%%%%%%%%%%%%%%%%%%%%%%%%%%%%%%%%TTTTTTTTTTTTTTT%%%%%%%%%%%%%%%%%%%%%%%%%%%%%%%%%%%%%%%%%%%%%%%%%%%%%%%
\begin{theorem}\label{lw-Eg}
Under Assumption~\eqref{B(x)}, if $H(\kappa)$ satisfies \eqref{cond-H}, $(\psi,\Ab)\in H^{1}(\Omega,\C)\times H^{1}_{\rm div}(\Omega)$ is a minimizer of \eqref{eq-2D-GLf} and $\overline {\mathcal{D}}\subset\Omega$ is open, then, 
$$\mathcal{E} (\psi,\Ab;\mathcal{D})\geq \kappa^{2}\int_{\mathcal{D}} \hat{f}\left(\frac{H}{\kappa}|B_{0}(x)|\right)\,dx+\textit{o}\left(\kappa H\ln\frac{\kappa}{H}\right)\,,\qquad{\rm as}\,\kappa\longrightarrow+\infty\,.$$
\end{theorem}
\begin{proof}
The proof is similar to the one of Theorem~\ref{up-Eg}.\\
Let
\begin{equation}\label{def-D-supB}
\mathcal{D}_{\ell,\rho}=\text{int}\left(\displaystyle{\cup_{\gamma\in \mathcal
I_{\ell,\rho}}}\overline{Q_{\gamma,\ell}}\right)
\end{equation}
and
\begin{equation}\label{supB}
\overline{B}_{\gamma,\ell}=\displaystyle\sup_{x\in Q_{\gamma,\ell}}|B_{0}(x)|\,,
\end{equation}
where $\mathcal I_{\ell,\rho}$ was introduced in \eqref{def-card-squares}.\\
If $(\psi,\Ab)$ is a minimizer of \eqref{eq-2D-GLf}, we have
$$\mathcal{E}(\psi,\Ab;\mathcal{D})=\mathcal{E}_{0}(\psi,\Ab;\mathcal{D}_{\ell,\rho})+\mathcal{E}_{0}(\psi,\Ab;\mathcal{D}\setminus \mathcal{D}_{\ell,\rho})+(\kappa H)^{2}\int_{\Omega}|\curl\Ab-B_{0}|^{2}\,dx\,,$$
where $\mathcal{E}_{0}(\psi,\Ab;\mathcal{D})$ is introduced in \eqref{eq-GLe0}.\\
Since the magnetic energy term and the energy in $\mathcal{D}\setminus \mathcal{D}_{\ell,\rho}$  are positive, we may write,
\begin{equation}\label{eD>eDl}
\mathcal{E}(\psi,\Ab;\mathcal{D})\geq \mathcal{E}_{0}(\psi,\Ab;\mathcal{D}_{\ell,\rho})\,.
\end{equation}
To estimate $\mathcal{E}_{0}(\psi,\Ab;\mathcal{D}_{\ell,\rho})$, we notice that,
$$\mathcal{E}_{0}(\psi,\Ab;\mathcal{D}_{\ell,\rho})=\sum_{\gamma\in\mathcal
I_{\ell,\rho}}\mathcal{E}_{0}(\psi,\Ab;Q_{\gamma,\ell})\,.$$
Recall that for any $\widetilde{\gamma}\in\overline{Q_{\gamma,\ell}}$ we have $B_{0}(\widetilde{\gamma})$ satisfies \eqref{F-A}. Then, we select $\widetilde{\gamma}$ such that 
$$|B_{0}(\widetilde{\gamma})|=\overline{B}_{\gamma,\ell}\,.$$ 
Using \eqref{est-loc-E0}, similarly as we did in the upper bound we recognize the upper Riemann sum, and get
\begin{equation}\label{sum-E0}
\mathcal{E}_{0}(\psi,\Ab;\mathcal{D}_{\ell,\rho})\geq \kappa^{2}(1-\delta)\int_{\mathcal{D}_{\ell,\rho}}\hat{f}\left(\frac{H}{\kappa}|B_{0}(x)|\right)\,dx-C\left( \frac{\kappa}{\ell}+\delta^{-1}(\kappa H)^{2}\ell^{4}+\delta^{-1}\kappa^{2}H^{2}\lambda^{2}\ell^{2\alpha} \right)\,.
\end{equation}
Notice that using the regularity of $\partial \mathcal{D}$ and \eqref{B(x)}, there exists $C>0$ such that
\begin{equation}\label{D-D_ell}
\forall\ell\in(0,1),\,\forall\rho\in(0,1),\quad|\mathcal{D}\setminus \mathcal{D}_{\ell,\rho}|\leq C(\ell+\rho)\,.
\end{equation}
We get by using property~(3) of $f$ in Theorem~\ref{pro-f(b)}, Assumption~\eqref{cond-H} and for some choice of $\rho$ to be determined later
\begin{align}\label{D-Dl}
\int_{\mathcal{D}_{\ell,\rho}}\hat{f}\left(\frac{H}{\kappa}|B_{0}(x)|\right)\,dx&\geq \int_{\mathcal{D}}\hat{f}\left(\frac{H}{\kappa}|B_{0}(x)|\right)\,dx-\int_{\mathcal{D}\setminus \mathcal{D}_{\ell,\rho}}\hat{f}\left(\frac{H}{\kappa}|B_{0}(x)|\right)\,dx\nonumber\\
&\geq \int_{\mathcal{D}}\hat{f}\left(\frac{H}{\kappa}|B_{0}(x)|\right)\,dx-C\frac{H}{\kappa}|\mathcal{D}\setminus \mathcal{D}_{\ell,\rho}|\,.
\end{align}
This implies that
\begin{equation}\label{fianl-Eg}
\E0(\kappa,H)\geq \kappa^{2}(1-\delta)\int_{\mathcal{D}}\hat{f}\left(\frac{H}{\kappa}|B_{0}(x)|\right)\,dx-r'(\kappa,H,\ell)\,,
\end{equation}
where $$r'(\kappa,H,\ell)=\mathcal{O}\left(\kappa H\ell+\kappa H\rho+\frac{\kappa}{\ell}+\delta^{-1}(\kappa H)^{2}\ell^{4}+\delta^{-1}\kappa^{2}H^{2}\lambda^{2}\ell^{2\alpha}\right)\,.$$
Having in mind \eqref{b-lambda}, then, the remainder term becomes
$$r'(\kappa,H,\ell)=\mathcal{O}\left(\kappa H\ell+\kappa H\rho+\frac{\kappa}{\ell}+\delta^{-1}(\kappa H)^{2}\ell^{4}+\delta^{-1}\kappa H\ln\frac{C_{0}\kappa}{H}\ell^{2\alpha}\right)\,.$$
The choice of the parameters $\delta$ in \eqref{choice-delta-alpha} and $\rho$, $\ell$ in \eqref{choice-ell-rho}  implies all error terms to be of lower order compared with  $\kappa H\ln\frac{\kappa}{H}$. This finishes the proof of Theorem~\ref{lw-Eg}.
\end{proof}
%%%%%%%%%%%%%%%%%%%%%%%%%%%%%%%%%%%%%%%%RRRRRRRRRRRRRRRRRRRRR%%%%%%%%%%%%%%%%%%%%%%%%%%%%%%%%%%%%%%%%%%%%%
\begin{rem}\label{lower-bound-local-energy-in-D}
Notice that $\mathcal{E}_{0}(\psi,\Ab;\mathcal{D})\geq\mathcal{E}_{0}(\psi,\Ab;\mathcal{D}_{\ell,\rho})$. Using \eqref{sum-E0} and \eqref{D-Dl} with the same choices of $\delta$, $\rho$ and $\ell$ as in Remark~\ref{choice-ell1}, we obtain
\begin{equation}
\mathcal{E}_{0}(\psi,\Ab;\mathcal{D})\geq  \kappa^{2}\int_{\mathcal{D}}\hat{f}\left(\frac{H}{\kappa}|B_{0}(x)|\right)\,dx+\textit{o}\left(\kappa H\ln\frac{\kappa}{H}\right)\,.
\end{equation}
Moreover, we can replace $\mathcal{D}$ by $\Omega$ and get
\begin{equation}
\mathcal{E}_{0}(\psi,\Ab;\Omega)\geq  \kappa^{2}\int_{\Omega}\hat{f}\left(\frac{H}{\kappa}|B_{0}(x)|\right)\,dx+\textit{o}\left(\kappa H\ln\frac{\kappa}{H}\right)\,.
\end{equation}
\end{rem}
%%%%%%%%%%%%%%%%%%%%%%%%%%%%%%%%%%CCCCCCCCCCCCCCC%%%%%%%%%%%%%%%%%%%%%%%%%%%%%%%%%%%%%%%%%%%%%%%%%%%%%%%%%
{\bf Proof of Corollary~\ref{estimate-magnetic-energy}}

Having in mind \eqref{eq-GLe0}, we write $$\mathcal{E}(\psi,\Ab;\Omega)=\mathcal{E}_{0}(\psi,\Ab;\Omega)+(\kappa H)^{2}\int_{\Omega}|\curl\Ab-B_{0}|^{2}\,dx\,.$$ Using the estimate of $\mathcal{E}(\psi,\Ab;\Omega)$ in Theorem~\ref{thm-2D-main}, we get, as $\kappa\longrightarrow+\infty$
\begin{equation}\label{cor-up}
\mathcal{E}_{0}(\psi,\Ab;\Omega)+(\kappa H)^{2}\int_{\Omega}|\curl\Ab-B_{0}|^{2}\,dx\leq \kappa^{2}\int_{\Omega}\hat{f}\left(\frac{H}{\kappa}|B_{0}(x)|\right)\,dx+\textit{o}\left(\kappa H\ln\frac{\kappa}{H}\right)\,.
\end{equation}
Remark~\ref{lower-bound-local-energy-in-D} tells us that
$$
\mathcal{E}_{0}(\psi,\Ab;\Omega)\geq  \kappa^{2}\int_{\Omega}\hat{f}\left(\frac{H}{\kappa}|B_{0}(x)|\right)\,dx+\textit{o}\left(\kappa H\ln\frac{\kappa}{H}\right)\,.$$
Therefore, \eqref{cor-up} becomes
\begin{multline}
\kappa^{2}\int_{\Omega}\hat{f}\left(\frac{H}{\kappa}|B_{0}(x)|\right)\,dx+\textit{o}\left(\kappa H\ln\frac{\kappa}{H}\right)+(\kappa H)^{2}\int_{\Omega}|\curl\Ab-B_{0}|^{2}\,dx\\
\leq \kappa^{2}\int_{\Omega}\hat{f}\left(\frac{H}{\kappa}|B_{0}(x)|\right)\,dx+\textit{o}\left(\kappa H\ln\frac{\kappa}{H}\right)\,.
\end{multline}
By simplification, we get \eqref{eq-estimate-magnetic-energy}.

\section{Proof of Theorem~\ref{th-loc-enrgy}: upper bound}\label{section6}
One aim of this section is to derive a sharp estimate of $\mathcal E_0(\psi,\Ab;Q_{\ell}(x_{0}))$, when $(\psi,\Ab)\in H^1(\Omega;\C)\times H^1_{\Div}(\Omega)$ is a minimizer of \eqref{eq-2D-GLf}.

The proof of the next proposition is similar to that in \cite[Proposition~6.2]{KA}, by replacing  $\displaystyle\frac{1}{R}$ by $\displaystyle\frac{b^{\frac{1}{2}}}{R}$.
%%%%%%%%%%%%%%%%%%%%%%%%%%%%%%%%%%%%%%%%%%%%PPPPPPPPPPPPPPPPP%%%%%%%%%%%%%%%%%%%%%%%%%%%%%%%%%%%%%%%%%%%%%%%%
\begin{prop}\label{prop-ub}
For $\alpha\in(0,1)$, there exist positive constants $C$ and $\kappa_{0}$ such that if $\kappa\geq\kappa_{0}$, $\ell \in (0,\frac 12)$, $\delta\in(0,1)$, $\rho>0$, $\ell^{2}\kappa H\rho\geq1$, $(\psi,\Ab)\in H^1(\Omega;\C)\times H^1_{\Div}(\Omega)$ is
a minimizer of \eqref{eq-2D-GLf}, and $ (\ell,x_0,\widetilde x_0)$ a $\rho$-admissible triple, then,
\begin{equation}\label{equ-ub}
\frac1{|Q_{\ell}(x_0)|}\mathcal E_0(\psi,\Ab;Q_{\ell}(x_0))\leq (1+\delta)\kappa^2 \hat{f}\left(\displaystyle\frac{H}{\kappa}|B_{0}(\widetilde{x}_{0})|\right)+C\left(\frac{\kappa}{\ell}+\delta^{-1}\ell^{4}\kappa^{2}H^{2}+\delta^{-1}\kappa^{2}H^{2}\lambda^{2}\ell^{2\alpha}\right)\,,
\end{equation}
where $\lambda$ is introduced in \eqref{lamda}.
\end{prop}
%%%%%%%%%%%%%%%%%%%%%%%%%%%%%%%%%%%%%RRRRRRRRRRRRR%%%%%%%%%%%%%%%%%%%%%%%%%%%%%%%%%%%%%%%%%%%%%%%%%%%%%%%%%
\begin{rem}\label{est-error}
Under Assumption~\eqref{cond-H}, with the choices of $\ell$, $\rho$ in \eqref{choice-ell-rho} and $\delta$ in \eqref{choice-delta-alpha}, we get that the error terms in \eqref{equ-ub} are of order $\kappa H\,\ln\frac{\kappa}{H}$
\end{rem}
%%%%%%%%%%%%%%%%%%%%%%%%%%%%%%%%%%%%%%%%%%%CCCCCCCCCCCCCCCCCCCCCC%%%%%%%%%%%%%%%%%%%%%%%%%%%%%%%%%%%%%%%
\begin{prop}\label{prop-ub2}
For any $\alpha\in(0,1)$, there exist positive constants $\widehat C_\alpha $ and $\kappa_{0}$ such that if $\kappa\geq\kappa_{0} $, $H(\kappa)$ satisfies \eqref{cond-H}, $\ell$ is chosen as in \eqref{choice-ell-rho}, $\delta$ as in \eqref{choice-delta-alpha}, $\ell^{2}\kappa H\rho\geq1$, $(\psi,\Ab)$ is a minimizer of \eqref{eq-2D-GLf}, and $ (\ell,x_0,\widetilde x_0)$ a $\rho$-admissible triple, then
\begin{multline}\label{est-A0-final}
\left|\frac{1}{|Q_{\ell}(x_0)|}\mathcal{E}_{0}(\psi,\sigma_{\ell}|B_{0}(\widetilde{x}_{0})|\Ab_{0}(x-x_0)+\nabla(\varphi_{x_{0},\widetilde{x}_{0}}+\phi_{x_{0}}),Q_{\ell}(x_0))-\kappa^{2}\hat{f}\left(\frac{H}{\kappa}|B_{0}(\widetilde{x}_{0})|\right)\right|\\
\leq \widehat C_\alpha \, \kappa H\,\left(\ln\frac{\kappa}{H}\right)^{\frac{\alpha}{4}} \,,
\end{multline}
where $\Ab_{0}$ is the magnetic potential introduced in \eqref{eq-hc2-mpA0}, $\sigma_{\ell}$ denotes the sign of $B_{0}$, $\phi_{x_{0}}$ is defined in \eqref{alpha} and $\varphi_{x_{0},\widetilde{x}_{0}}$ is the function satisfying \eqref{F-A}.
\end{prop}
\begin{proof}~\\
\textbf{Lower bound:}
We refer to \eqref{final-maj} and \eqref{def-bR}. We obtain
\begin{equation}\label{up-A03}
\frac{1}{|Q_{\ell}(x_0)|}\mathcal{E}_{0}(\psi,\sigma_{\ell}|B_{0}(\widetilde{x}_{0})|\Ab_{0}(x-x_0)+\nabla(\varphi_{x_{0},\widetilde{x}_{0}}+\phi_{x_{0}}),Q_{\ell}(x_0))\geq \kappa^{2}\hat{f}\left(\frac{H |B_{0}(\widetilde{x}_{0})|}{\kappa}\right)-C \frac{\kappa}{\ell}\,,
\end{equation}
where $C$ is a positive constant.\\
If we select $\ell$ as in \eqref{choice-ell-rho}, we get 
\begin{equation}\label{up-A03-final}
\frac{1}{|Q_{\ell}(x_0)|}\mathcal{E}_{0}(\psi,\sigma_{\ell}|B_{0}(\widetilde{x}_{0})|\Ab_{0}(x-x_0)+\nabla(\varphi_{x_{0},\widetilde{x}_{0}}+\phi_{x_{0}}),Q_{\ell}(x_0))\geq \kappa^{2}\hat{f}\left(\frac{H |B_{0}(\widetilde{x}_{0})|}{\kappa}\right)-C \,(\kappa^{5}H)^\frac{1}{4}\,.
\end{equation}
Assumption \eqref{cond-H} permits to verify that the remainder $(\kappa^{5}H)^\frac{1}{4} = \mathcal O (\kappa H (\ln \frac \kappa H)^\frac \alpha 4 )$.\\
\textbf{Upper bound:} 
Collecting \eqref{eq-e(F)<e(A)} and \eqref{equ-ub}, we get for any $\alpha\in(0,1)$, the existence of $C'>0$ such that
\begin{multline}\label{lw-A0}
\frac{1}{|Q_{\ell}(x_0)|}\mathcal{E}_{0}(\psi,\sigma_{\ell}|B_{0}(\widetilde{x}_{0})|\Ab_{0}(x-x_0)+\nabla(\varphi_{x_{0},\widetilde{x}_{0}}+\phi_{x_{0}}),Q_{\ell}(x_0))\leq \kappa^2\hat{f}\left(\displaystyle\frac{H}{\kappa}|B_{0}(\widetilde{x}_{0})|\right)\\
+C'\left(\frac{\kappa}{\ell}+\delta^{-1}\ell^{4}\kappa^{2}H^{2}+\delta^{-1}\kappa^{2}H^{2}\lambda^{2}\ell^{2\alpha}\right)\,,
\end{multline}
where  $\lambda$ is introduced in \eqref{lamda}.\\
Using \eqref{b-lambda} and selecting  $\ell$ as in \eqref{choice-ell-rho} and $\delta$ as in \eqref{choice-delta-alpha}, we get the existence of a  constant $C_{\alpha}$ such that
\begin{multline}\label{lw-A0-final}
\frac{1}{|Q_{\ell}(x_0)|}\mathcal{E}_{0}(\psi,\sigma_{\ell}|B_{0}(\widetilde{x}_{0})|\Ab_{0}(x-x_0)+\nabla(\varphi_{x_{0},\widetilde{x}_{0}}+\phi_{x_{0}}),Q_{\ell}(x_0))\leq \kappa^2\hat{f}\left(\displaystyle\frac{H}{\kappa}|B_{0}(\widetilde{x}_{0})|\right)\\
+C_{\alpha}\kappa H\left(\ln\frac{\kappa}{H}\right)^{\frac{\alpha}{4}}\,.
\end{multline}
This achieves the proof of the lemma.
\end{proof}
The next lemma will be useful in the proof of Theorem~\ref{th-loc-enrgy}.
%%%%%%%%%%%%%%%%%%%%%%%%%%%%%%%%%%%%%%%%%%%%LLLLLLLLLLLLLLL%%%%%%%%%%%%%%%%%%%%%%%%%%%%%%%%%%%%%%%%%%%%%%%%%
\begin{lem}\label{est-gamma} For any $C_1>0$, there exist positive constants $C$ and $\kappa_{0}$ such that if $\ell\in (0,1)$, $\kappa_{0}\leq \kappa$ and $(\psi,\Ab)\in H^{1}(\Omega;\C)\times H^{1}_{\rm div}(\Omega)$ is a solution of \eqref{eq-2D-GLeq}, then 
\begin{equation}
\int_{\mathcal V_\ell (\Gamma, C_1)}|(\nabla-i\kappa H \Ab)\psi|^{2}\,dx\leq C\kappa^{2}\ell\left(1+\frac{1}{\kappa\ell^{\frac{3}{2}}}\right)\,,
\end{equation}
where  
$$ 
\mathcal V_\ell (\Gamma, C_1) = \left\{x\in\Omega:{\rm dist}(x,\Gamma)\leq C_{1}\ell  \mbox{ and } d(x,\partial \Omega) \geq \frac{\ell }{C_1}\right\}\,.
$$
\end{lem}
\begin{proof} 
Using \eqref{est-grad} and the fact that the range of $\hat{f}$ is the interval $[0,1/2]\,$, we get
\begin{equation}\label{pr21}
\|(\nabla-i\kappa H\Ab)\psi\|_{L^{2}(\Omega)} \leq C \kappa\,.
\end{equation}
Hence the improvment given by the lemma is when   $\frac 1C  \kappa^{-2} \leq \ell \leq \ell_0\,$. \\  Let $C_2 > C_1$ and for $\ell$ small enough
we define the following sets $\mathcal D_\ell^1= \mathcal V_\ell (\Gamma, C_1) \, $ and \break 
$\mathcal D_\ell^2=\mathcal V_\ell (\Gamma, C_2)$\,.
We can construct a cut-off function $\chi_{\ell}\in C^{\infty}_{c}(\Omega)$ such that
\begin{equation}\label{chi-l}
0\leq \chi_{\ell} \leq 1~{\rm in}~\R^{2}\,,\quad {\rm supp}\chi_{\ell}\subset \mathcal D_\ell^2\subset  \subset \Omega \,,\quad\chi_{\ell}=1~{\rm in}~\mathcal D_\ell^1\quad{\rm and}\quad|\nabla\chi_{\ell}|\leq\frac{C}{\ell}~{\rm in}~\R^{2}\,,
\end{equation}
where $C$ is a positive constant independent of $\ell$.\\
The minimizer $\psi$ satisfies 
\begin{equation}
\kappa^2\psi(1-|\psi|^2)=-(\nabla-i\kappa H\Ab)^2\psi\quad{\rm in}\,\Omega\,.
\end{equation}
 We multiply the above equation by $ \chi_{\ell} \bar \psi$, it results from an integration by parts that
\begin{align}\label{pr1}
\kappa^{2}\int_{\mathcal D_\ell^2} \chi_{\ell}(1-|\psi|^{2})|\psi|^{2}\,dx&=\int_{\mathcal D_\ell^2} (\nabla-i\kappa H \Ab)\psi\; \overline{\chi_{\ell}(\nabla-i\kappa H\Ab)\psi+\psi\nabla\chi_{\ell}}\,dx\nonumber\\
&=\int_{\mathcal D_\ell^2} \chi_{\ell}\, |(\nabla-i\kappa H \Ab)\psi|^{2}\,dx+\int_{\mathcal D_\ell^2} \nabla\chi_{\ell}\,\bar \psi\, (\nabla-i\kappa H \Ab)\psi\,dx\,.
\end{align}
 Using H\"older inequality, we have
\begin{equation}\label{pr2}
\left| \int_{\mathcal{D}_{\ell}^{2} } \nabla\chi_{\ell}\,\bar \psi \, (\nabla-i\kappa H \Ab)\psi\,dx\right|\leq \|(\nabla-i\kappa H\Ab)\psi\|_{L^{2}(\Omega,\mathbb C^2)}\, \| |\nabla\chi_{\ell}|\, \psi\|_{ L^{2} (\mathcal D_\ell^2)  }\,.
\end{equation}
Notice that $|\mathcal D_\ell^2|\leq C'\ell$. Using \eqref{chi-l} and the bound $\|\psi\|_{\infty}\leq 1$, we obtain
\begin{equation}\label{pr22}
\|| \nabla\chi_{\ell}|\, \psi\|_{L^{2}({\mathcal D_\ell^2})}\leq C''\ell^{-\frac{1}{2}}\,.
\end{equation}
Putting \eqref{pr21} and \eqref{pr22} into \eqref{pr2}, we get
\begin{equation}
\left| \int_{\mathcal D_\ell^2} \nabla\chi_{\ell}\,\bar \psi(\nabla-i\kappa H \Ab)\psi\,dx\right|  \leq C\,\kappa\,\ell^{-\frac{1}{2}}\,,
\end{equation}
and consequently
\begin{equation}
\int_{\mathcal{D}_{\ell}^{2} } \chi_{\ell}|(\nabla-i\kappa H \Ab)\psi|^{2}\,dx\leq \kappa^{2}\int_{\mathcal D_\ell^2} \chi_{\ell}(1-|\psi|^{2})|\psi|^{2}\,dx+C\,\kappa\,\ell^{-\frac{1}{2}}\,.
\end{equation}
The lemma easily follows from the control of the area of $\mathcal D_{\ell}^2$ and from observing that $\chi_\ell=1$ on $\mathcal D_{\ell}^1$\,.
\end{proof}
\begin{rem}\label{rem6.5}
We get a similar estimate by replacing  in the lemma $\Gamma$ by the boundary $\partial \mathcal D$ of a regular open set $\mathcal D$ compactly contained in $\Omega$.
\end{rem}
%%%%%%%%%%%%%%%%%%%%%%%%%%%%%%%%%%%%%%%%%%%%%RRRRRRRRRRRRRR%%%%%%%%%%%%%%%%%%%%%%%%%%%%%%%%%%%%%%%%%%%%%%
%\begin{rem}\label{up-e-D/D'}
%Keeping the same choise of $\ell$  as in \eqref{choice-ell-rho}. In order to obtain the term $\kappa^{2}\ell$ in \eqref{prb-l} of order $\kappa H\ln\frac{\kappa}{H}$, we need a stronger condition than \eqref{cond-H} on $H(\kappa)$. In fact, thanks to \eqref{cond-H2}, we have as $\kappa\longrightarrow+\infty$
%$$\frac{\kappa^{2}\ell}{\kappa H\ln\frac{\kappa}{H}}=\left(\frac{\kappa^{3}}{H^{5}}\right)^{\frac{1}{4}} \frac{1}{\ln\frac{\kappa}{H}}\ll1\,,$$
%and
%\begin{equation}
%\int_{\Gamma}|(\nabla-i\kappa H \Ab)\psi|^{2}\,dx=\textit{o}\left(\kappa H\ln\frac{\kappa}{H}\right)\,.
%\end{equation}
%\end{rem}
%%%%%%%%%%%%%%%%%%%%%%%%%%%%%%%%%%%%%%%%%%%%%%%%%%%%%%%%%%%%%%%%%%%%%%%%%%%%%%%%%%%%%%%%%%%%%%%%%%%%%%%%%
{\bf  End of the proof of Theorem~\ref{th-loc-enrgy}.\\}
The proof of \eqref{lw-loc-enrgy}  is already obtained in Theorem~\ref{lw-Eg}. Hence it remains only to give the proof of \eqref{up-loc-enrgy}.\\
We keep the same notation as in \eqref{inf-B}, \eqref{def-card-squares} and \eqref{def-D-supB}. If $(\psi,\Ab)$ is a minimizer of \eqref{eq-2D-GLf}, we start with \eqref{eq-GLen-D} and write,
\begin{equation}\label{eq-loc-en}
\mathcal E(\psi,\Ab;\mathcal{D})=\mathcal E_0(\psi,\Ab;\mathcal{D}_{\ell,\rho})+\mathcal E_0(\psi,\Ab;\mathcal{D}\setminus \mathcal{D}_{\ell,\rho}) + (\kappa H)^2 \int_{\Omega} |\curl\big(\Ab - \Fb\big)|^2\,dx \,.
\end{equation}
To estimate $\mathcal E_0(\psi,\Ab;\mathcal{D}_{\ell,\rho})$, we notice that,
$$\mathcal E_0(\psi,\Ab;\mathcal{D}_{\ell,\rho})=\sum_{\gamma\in\mathcal I_{\ell,\rho}}\mathcal E_0(\psi,\Ab;Q_{\gamma,\ell})\,.$$
Remark~\ref{est-error} tells us that the error terms in \eqref{equ-ub} are of order $\kappa H\,\ln\frac{\kappa}{H}$. Therfore, using \eqref{equ-ub}, we get
$$
\mathcal E_0(\psi,\Ab;\mathcal{D}_{\ell,\rho})\leq \kappa^2\sum_{\gamma\in\mathcal I_{\ell,\rho}}\hat{f}\left(\displaystyle\frac{H}{\kappa}|B_{0}(\widetilde{\gamma})|\right)\ell^{2}+\textit{o}\left(\kappa H\ln\frac{\kappa}{H}\right)\,,\qquad\text{as}~\kappa\longrightarrow+\infty\,.
$$
We select $\widetilde{\gamma}\in \overline{Q_{\ell}(\gamma)}$ such that $|B_{0}(\widetilde{\gamma})|=\underline{B}_{\gamma,\ell}$, where $\underline{B}_{\gamma,\ell}$ is defined in \eqref{inf-B}. By monotonicity of $\hat{f}$, $\hat{f}$ is Riemann-integrable and its integral is larger than any lower Riemann sum. Thus
\begin{equation}\label{up-e-Drho}
\mathcal E_0(\psi,\Ab;\mathcal{D}_{\ell,\rho})\leq \kappa^2\int_{\mathcal{D}_{\ell,\rho}} \hat{f}\left(\displaystyle\frac{H}{\kappa}|B_{0}(x)|\right)\,dx+\textit{o}\left(\kappa H\ln\frac{\kappa}{H}\right)\,,\qquad\text{as}~\kappa\longrightarrow+\infty\,.
\end{equation}
Moreover, recalling that $\hat f$ is a positive function and $\mathcal{D}_{\ell,\rho}\subset \mathcal{D}$, \eqref{up-e-Drho} becomes
\begin{equation}\label{up-e-D}
\mathcal E_0(\psi,\Ab;\mathcal{D}_{\ell,\rho})\leq \kappa^2\int_{\mathcal{D}} \hat{f}\left(\displaystyle\frac{H}{\kappa}|B_{0}(x)|\right)\,dx+\textit{o}\left(\kappa H\ln\frac{\kappa}{H}\right)\,,\qquad\text{as}~\kappa\longrightarrow+\infty\,.
\end{equation}
For estimating  $\mathcal E_0(\psi,\Ab;\mathcal{D}\setminus \mathcal{D}_{\ell,\rho})$, we use  Lemma~\ref{est-gamma}, Remark \ref{rem6.5} and we keep the same choice of $\ell$ and $\rho$ as in  \eqref{choice-ell-rho}, which implies $\rho\ll\ell$, we obtain  that
\begin{equation}
\int_{\mathcal{D}\setminus \mathcal{D}_{\ell,\rho}}|(\nabla-i\kappa H \Ab)\psi|^{2}\,dx\leq C(\,\kappa^{2}\,\ell+\,\kappa\,\ell^{-\frac{1}{2}})\,.
\end{equation}
Adding the second term in the energy leads to
\begin{equation}\label{up-e-Drho2}
\mathcal E_0(\psi,\Ab;\mathcal{D}\setminus \mathcal{D}_{\ell,\rho})\leq C(\,\kappa^{2}\,\ell+\,\kappa\,\ell^{-\frac{1}{2}})\,.
\end{equation}
The second term in the right hand side is controlled by the first one if
$$
\kappa \ell^\frac 32\gg1\,.
$$
This is effectively satisfied with our choice of $\ell$ and the condition on $H(\kappa)$.\\
 In order to obtain the term $\kappa^{2}\ell$ in \eqref{up-e-Drho2} comparatively small with  $\kappa H\ln\frac{\kappa}{H}$, we need a stronger condition than \eqref{cond-H} on $H(\kappa)$. In fact, we have
$$\frac{\kappa^{2}\ell}{\kappa H\ln\frac{\kappa}{H}}=\left(\frac{\kappa^{3}}{H^{5}}\right)^{\frac{1}{4}} \frac{1}{\ln\frac{\kappa}{H}}\,,$$
and thanks to \eqref{cond-H2},  as $\kappa\longrightarrow+\infty\,,$
$$\frac{1}{\ln\frac{\kappa}{H}}\ll1\quad{\rm and}\quad \frac{\kappa^{3}}{H^{5}}\leq C\,,$$
where $C$ is a positive constant.\\
This implies that
$$
\kappa^{2}\ell=\textit{o}\left(\kappa H \ln\frac{\kappa}{H}\right)\,,
$$
and consequently
\begin{equation}\label{up-e-Drho3}
\mathcal E_0(\psi,\Ab;\mathcal{D}\setminus \mathcal{D}_{\ell,\rho})=\textit{o}\left(\kappa H\ln\frac{\kappa}{H}\right)\,.
\end{equation}
Corollary~\ref{estimate-magnetic-energy} tells us that, under Assumption~\eqref{cond-H},
\begin{equation}\label{up-mag-en}
(\kappa H)^{2}\int_{\Omega}|\curl\Ab-B_{0}|^{2}\,dx=\textit{o}\left(\kappa H \ln\frac{\kappa}{H}\right)\,,\qquad\text{as}~\kappa\longrightarrow+\infty\,.
\end{equation}
Therefore, by collecting \eqref{up-e-D}, \eqref{up-e-Drho3} and \eqref{up-mag-en} and inserting then into \eqref{eq-loc-en}, we finish the proof of \eqref{up-loc-enrgy}.
%%%%%%%%%%%%%%%%%%%%%%%%%%%%%%%%%%%%%%%%%%%%%%%%%%%%%%%%%%%%%%%%%%%%%%%%%%%%%%%%%%%%%%%%%%%%%%%%%%%%%%%%%
\section{Vortices and concentration of the energy}\label{section7}
This section is devoted to the proof of Theorem~\ref{distribution-vortices}.  We keep the choice of $\ell$ given in \eqref{choice-ell-rho}:
$$
\ell=(\kappa H)^{-\frac{1}{4}}\,,
$$
but  we select $\rho$ and $\alpha$ as follows:
\begin{equation}\label{rho2}
\rho=\left(\ln\frac{\kappa}{H}\right)^{-\frac{1}{2}}\,,\qquad \alpha=\frac{1}{2}\,.
\end{equation}
\subsection{Energy in a $\rho$-admissible box}
If $(\psi,\Ab)\in H^{1}(\Omega;\C)\times H^{1}_{\rm div}(\Omega;\R)$, we consider for all $\rho$-admissible pair $(\ell, x_{0})$ the local energy
in $Q_\ell(x_0)$:
$$
\mathcal{E}_{0}(\psi,\Ab,Q_{\ell}(x_0))=\int_{Q_{\ell}(x_0)}\left(|(\nabla-i\kappa H\Ab)\psi|^{2}+\frac{\kappa^{2}}{2}(1-|\psi|^{2})^{2}\right)\,dx\,.
$$
\subsection{Division of the square $Q_{\ell}(x_0)$}~\\
Let $H=H(\kappa)$ be a function satisfying \eqref{cond-H}. For reasons that will become clear in Proposition~\ref{serfaty}, we need to divide $Q_{\ell}(x_0)$ into $\mathcal N= M^{2}$ disjoint open squares  $(Q^{j}_{\delta(\kappa)})_{j\in \mathcal{J}}$ such that 
$$Q_{\ell}(x_0)=\cup_{j\in \mathcal{J}}\overline{Q^{j}_{\delta(\kappa)}}\,,$$
with
\begin{equation}\label{def-M}
M=\left[2^{\frac{7}{8}}\,(\kappa H)^{\frac{1}{4}}\,\left(\ln\frac{\kappa}{H}\right)^{-\frac{7}{8}} \right]\,,
\end{equation}
where for $t \in\R$, $[t]$ denotes the integer part of $t$.\\
The side length of theses squares is consequently
\begin{equation}\label{delta(kappa)}
\delta(\kappa)=\frac{\ell}{M}\sim 2^{-\frac{7}{8}}(\kappa H)^{-\frac{1}{2}}\left(\ln\frac{\kappa}{H}\right)^{\frac{7}{8}}\,.
\end{equation}
Let us introduced for all $\rho$-admissible triple $(\ell,x_0,\widetilde x_0)$ the functions $b$ and $R$ by
\begin{equation}\label{def-bR3}
R(\kappa,H,\widetilde{x}_{0})=2^{-\frac{7}{8}}\left(\ln\frac{\kappa}{H}\right)^{\frac{7}{8}}|B_{0}(\widetilde{x}_{0})|^{\frac{1}{2}}\quad{\rm and}\quad b(\kappa,H,\widetilde{x}_{0})=\frac{H}{\kappa}|B_{0}(\widetilde{x}_{0})|\,.
\end{equation}
Notice that $b(\kappa,H,\widetilde{x}_{0})$ and $\displaystyle\frac{1}{R(\kappa,H,\widetilde{x}_{0})}$ are {\it uniformly $\textit{o}(1)$} as $\kappa\longrightarrow+\infty$, in the following sense:\\
 For all $\epsilon > 0$ there exists $\kappa_{0}>0$  such that $\forall\kappa\geq \kappa_{0}$, $H$ satisfying \eqref{cond-H}, $\rho$ introduced in \eqref{rho2} and any $\rho$-admissible triple  $(\ell,x_0,\widetilde{x}_{0})$   $$ |b(\kappa,H,\widetilde{x}_{0})|+\frac{1}{R(\kappa,H,\widetilde{x}_{0})}<\epsilon\,.$$
In fact, we have as $\kappa\longrightarrow+\infty$
\begin{equation*} 
R(\kappa,H,\widetilde{x}_{0})\geq 2^{-\frac{7}{8}}\left(\ln\frac{\kappa}{H}\right)^{\frac{7}{8}}\rho^{\frac{1}{2}}\geq \frac 1C  \left(\ln\frac{\kappa}{H}\right)^{\frac{5}{8}} \gg 1\,.
\end{equation*}
Since $B_{0}\in C^{\infty}(\overline{\Omega})$, we have also  
\begin{equation}\label{asy-b}
0 < b(\kappa,H,\widetilde{x}_{0}) \leq  \frac{H}{\kappa} \overline{\beta}_0 \ll 1\,,
\end{equation}
where 
\begin{equation}\label{beta0}
\overline{\beta}_0= \sup_{x\in \overline{\Omega}} |B_0(x)|\,.
\end{equation}
 More precisely, let
\begin{equation}
b(\kappa,H,\widetilde{x}_{0})=\hat{b}(\kappa,H,\beta)\qquad{\rm and}\qquad R(\kappa,H,\widetilde{x}_{0})=\hat{R}(\kappa,H,\beta)\,,
\end{equation}
where $\beta=|B_{0}(\widetilde{x}_{0})|$.\\
We define the function:
\begin{equation}\label{h}
h(\kappa,H)=\max\left(\left(\ln\frac{\kappa}{H}\right)^{-\frac{3}{8}},\sup_{\overline{\beta}_0 \geq \beta \geq  \left(\ln\frac{\kappa}{H}\right)^{-\frac{1}{2}}}{\rm err} (\hat{b}(\kappa,H,\beta),\hat{R}(\kappa,H,\beta))\right)\,,
\end{equation} 
where ${\rm err}(b,R)$ is defined in Proposition~\ref{err}.\\
Notice that $h$ satisfies
\begin{equation}\label{app-ha}
h(\kappa,H)=\textit{o}(1)\,,\qquad{\rm as}~\kappa\longrightarrow+\infty\,.
\end{equation}

 Next, we will use a method introduced by E. Sandier and S. Serfaty in \cite{SS3}. We distinguish in the family indexed by $\mathcal J $ two types of squares respectively called the `\textit{nice squares}' $(Q^{j}_{\delta(\kappa)})$ 
 which are indexed in $\mathcal J^n$ and the  `\textit{bad squares}' $(Q^{j}_{\delta(\kappa)})$ indexed in $\mathcal J^b $. The set $\mathcal J^n $ is the set of indices $ j\in \mathcal J $ such that
\begin{multline}\label{nice-squares}
\mathcal{E}_{0}(\psi,\sigma_{\ell}|B_{0}(\widetilde{x}_{0})|\Ab_{0}(x-x_{0})+\nabla\varphi,Q^{j}_{\delta(\kappa)})\leq  \delta(\kappa)^{2}\kappa^{2}\hat{f}\left(\displaystyle\frac{H}{\kappa}|B_{0}(\widetilde{x}_{0})|\right)(1+\,h(\kappa,H)^{\frac{1}{2}}).
\end{multline}
The set $\mathcal J^b $ is the set of indices $ j\in \mathcal J $ such that
\begin{multline}\label{bad-squares}
\mathcal{E}_{0}(\psi,\sigma_{\ell}|B_{0}(\widetilde{x}_{0})|\Ab_{0}(x-x_{0})+\nabla\varphi,Q^{j}_{\delta(\kappa)})>\delta(\kappa)^{2}\kappa^{2}\hat{f}\left(\displaystyle\frac{H}{\kappa}|B_{0}(\widetilde{x}_{0})|\right)(1+h(\kappa,H)^{\frac{1}{2}}).
\end{multline}
Hence we have $\mathcal{J}=\mathcal J^n  \cup \mathcal J^b $. We denote by  $\mathcal{N}^{g}$  the cardinal of  $\mathcal J^n $ and by $\mathcal{N}^{b}$   the cardinal of $\mathcal J^b$.
%LLLLLLLLLLLLLLLLLLLLLLLLLLLLLLLLLLLLLLLLLLLLLLLLLLLLLLLLLLLLLLLLLLLLLLLLLLLLLLLLLLLLLLLLLLLLLLLLLLLLLLLLLLLLLLLLLLLLLLLLLLLLLLLLLLL%
\begin{lemma}\label{N'=o(N)}
There exist positive constants $C$ and $\kappa_{0}$ such that if  $\kappa\geq\kappa_{0}$, then
\begin{equation}\label{eq-N'=o(N)}
\mathcal N^b\leq C\,\frac{h(\kappa,H)^{\frac{1}{2}}}{1-h(\kappa,H)^{\frac{1}{2}}}\mathcal N^n\,,
\end{equation}
where $h$ is introduced in \eqref{h}.
\end{lemma}
\begin{proof}
Recall that  $\Ab_{0}$ is the magnetic potential introduced in \eqref{eq-hc2-mpA0}, $\phi_{x_{0}}$ is defined in \eqref{alpha} and that, for $\widetilde{x}_{0}\in \overline{Q_{\ell}(x_{0})}$, $\varphi_{x_{0},\widetilde{x}_{0}}$ is the function satisfying \eqref{F-A}.\\
Having in mind the definition of $b$ and $R$ in \eqref{def-bR3} and their properties, and using \eqref{eN<f}, we get from \eqref{E0<} and \eqref{F>} the following inequality
\begin{align}\label{up-E0-smal}
\mathcal{E}_{0}(\psi,\sigma_{\ell}|B_{0}(\widetilde{x}_{0})|\Ab_{0}(x-x_{0})+\nabla\varphi,Q^{j}_{\delta(\kappa)})&\geq\frac{e_{N}(b,R)}{b}\nonumber\\
&\geq \frac{R^{2}}{b}\hat{f}(b)\left(1-{\rm err}(b,R)\right)\,,
\end{align}
 where $\varphi=\phi_{x_{0}}+\varphi_{x_{0},\widetilde{x}_{0}}$, $e_{N}$ is introduced in \eqref{eN}, $b=b(\kappa,H,\widetilde{x}_{0})$ and $R=R(\kappa,H,\widetilde{x}_{0})$.\\
As a consequence of \eqref{h}, \eqref{up-E0-smal} becomes
\begin{equation}\label{up-E0-smal2}
\mathcal{E}_{0}(\psi,\sigma_{\ell}|B_{0}(\widetilde{x}_{0})|\Ab_{0}(x-x_{0})+\nabla\varphi,Q^{j}_{\delta(\kappa)})\geq \kappa^{2}\delta(\kappa)^{2}\hat{f}\left(\displaystyle\frac{H}{\kappa}|B_{0}(\widetilde{x}_{0})|\right)(1-\,h(\kappa,H))\,.
\end{equation}
 Notice that
\begin{equation}\label{area-square}
|Q_{\ell}(x_{0})| =\sum_{j\in\mathcal{J}}\left|Q^{j}_{\delta(\kappa)}\right| =(\mathcal N^n+\mathcal N^b)\, \delta(\kappa)^{2}\,.
\end{equation}
Thus, we may write
\begin{equation}\label{est-E0-smal}
\sum_{j\in\mathcal{J}}\mathcal{E}_{0}\left(\psi,\sigma_{\ell}|B_{0}(\widetilde{x}_{0})|\Ab_{0}(x-x_{0})+\nabla\varphi,Q^{j}_{\delta(\kappa)}\right)=\mathcal{E}_{0}(\psi,\sigma_{\ell}|B_{0}(\widetilde{x}_{0})|\Ab_{0}(x-x_{0})+\nabla\varphi,Q_{\ell}(x_0))\,.
\end{equation}
When \eqref{cond-H} is satisfied, we have uniformly
$$
\ell^{2}\kappa H\rho=(\kappa H)^{\frac{1}{2}}\,\left(\ln\frac{\kappa}{H}\right)^{-\frac{1}{2}}\gg1\,,\qquad{\rm as}~\kappa\longrightarrow+\infty\,.
$$
Hence the assumptions of Proposition~\ref{prop-ub2} are satisfied. Putting \eqref{est-E0-smal} into \eqref{est-A0-final}, using \eqref{rho2} and \eqref{area-square}, we get
\begin{multline}\label{app-main-term}
\left|\sum_{j\in\mathcal{J}}\left[\mathcal{E}_{0}\left(\psi,\sigma_{\ell}|B_{0}(\widetilde{x}_{0})|\Ab_{0}(x-x_{0})+\nabla\varphi,Q^{j}_{\delta(\kappa)}\right)-\kappa^{2}\delta(\kappa)^{2}\hat{f}\left(\displaystyle\frac{H}{\kappa}|B_{0}(\widetilde{x}_{0})|\right)\right]\right|\\
\leq C\,(\mathcal N^b+\mathcal N^n)\,\delta(\kappa)^{2}\,\kappa H\left(\ln\frac{\kappa}{H}\right)^{\frac{1}{8}}\,.
\end{multline}
Using the monotonicity of $\hat{f}$ and remembering that $ (\ell,x_0,\widetilde x_0)$ is a $\rho$-admissible triple, we get,
\begin{equation}\label{lb-f1}
\hat{f}\left(\frac{H}{\kappa}|B_{0}(\widetilde{x}_{0})|\right)\geq\hat{f}\left(\frac{H}{\kappa}\rho\right)\,.
\end{equation}
Using \eqref{main-f(b)}, \eqref{rho2} and \eqref{asy-b}, we obtain as $\kappa\longrightarrow+\infty$
\begin{align}\label{lb-f2}
\kappa^{2}\,\hat{f}\left(\frac{H}{\kappa}\rho\right)&\geq \kappa H\rho \left(\ln\frac{\kappa}{H \rho}\right)\left(1+\hat{s}\left(\frac{H}{\kappa}\rho\right)\right)\\
&\geq \kappa H\rho \left(\ln\frac{\kappa}{H}+\ln\frac{1}{\rho}\right)\left(1+\hat{s}\left(\frac{H}{\kappa}\rho\right)\right)\\
&\geq\kappa H\left(\ln\frac{\kappa}{H}\right)^{\frac{1}{2}}\left(1+\hat{s}_{1}(\kappa,H)\right)\,,
\end{align}
where $\hat{s}_{1}(\kappa,H)=\hat{s}\left(\frac{H}{\kappa}\left(\ln\frac{\kappa}{H}\right)^{-\frac{1}{2}}\right)$ is uniformly $\textit{o}(1)$ for $H$ satisfying \eqref{cond-H}.\\
Collecting \eqref{lb-f1}-\eqref{lb-f2}, we get for $\kappa$ sufficiently large
\begin{align*}
\kappa^{2}\,\hat{f}\left(\frac{H}{\kappa}|B_{0}(\widetilde{x}_{0})|\right)&\geq \kappa H\,\left(\ln\frac{\kappa}{H}\right)^{\frac{1}{2}}(1+\hat{s}_{1}(\kappa,H))\nonumber\\
&\geq \frac{\kappa H}{2}\,\left(\ln\frac{\kappa}{H}\right)^{\frac{1}{2}}\,.
\end{align*}
Multiplying both sides by $\left(\ln\frac{\kappa}{H}\right)^{-\frac{3}{8}}$ and using the definition of $h$ in \eqref{h}, we get
\begin{equation}\label{lw-f3}
\kappa^{2}\,\hat{f}\left(\frac{H}{\kappa}|B_{0}(\widetilde{x}_{0})|\right)\,h(\kappa,H)\geq \frac{\kappa H}{2}\,\left(\ln\frac{\kappa}{H}\right)^{\frac{1}{8}}\,.
\end{equation}
Puting \eqref{lw-f3} into \eqref{app-main-term}, we obtain
\begin{multline}\label{app-main-term2}
\left|\sum_{j\in\mathcal{J}}\left[\mathcal{E}_{0}\left(\psi,\sigma_{\ell}|B_{0}(\widetilde{x}_{0})|\Ab_{0}(x-x_{0})+\nabla\varphi,Q^{j}_{\delta(\kappa)}\right)-\kappa^{2}\delta(\kappa)^{2}\hat{f}\left(\displaystyle\frac{H}{\kappa}|B_{0}(\widetilde{x}_{0})|\right)\right]\right|\\
\leq C\,(\mathcal N^b+\mathcal N^n)\,\kappa^{2}\,\delta(\kappa)^{2}\,\hat{f}\left(\displaystyle\frac{H}{\kappa}|B_{0}(\widetilde{x}_{0})|\right)\,h(\kappa,H)\,.
\end{multline}
Using \eqref{bad-squares}, \eqref{up-E0-smal2} and \eqref{app-main-term2}, we may write
\begin{align*}
&\mathcal N^b\,\kappa^{2}\,\delta(\kappa)^{2}\,\hat{f}\left(\displaystyle\frac{H}{\kappa}|B_{0}(\widetilde{x}_{0})|\right)\,h(\kappa,H)^{\frac{1}{2}}\\
&\qquad\qquad\leq \sum_{j\in\mathcal J^b }\left[\mathcal{E}_{0}\left(\psi,\sigma_{\ell}|B_{0}(\widetilde{x}_{0})|\Ab_{0}(x-x_{0})+\nabla\varphi,Q^{j}_{\delta(\kappa)}\right)-\kappa^{2}\delta(\kappa)^{2}\hat{f}\left(\displaystyle\frac{H}{\kappa}|B_{0}(\widetilde{x}_{0})|\right)\right]\\
&\qquad\qquad\leq\sum_{j\in\mathcal{J}}\left[\mathcal{E}_{0}\left(\psi,\sigma_{\ell}|B_{0}(\widetilde{x}_{0})|\Ab_{0}(x-x_{0})+\nabla\varphi,Q^{j}_{\delta(\kappa)}\right)-\kappa^{2}\delta(\kappa)^{2}\hat{f}\left(\displaystyle\frac{H}{\kappa}|B_{0}(\widetilde{x}_{0})|\right)\right]\\
&\qquad\qquad\qquad\qquad\qquad\qquad\qquad\qquad\qquad\qquad\qquad+\,\sum_{j\in\mathcal J^n }\kappa^{2}\delta(\kappa)^{2}\hat{f}\left(\displaystyle\frac{H}{\kappa}|B_{0}(\widetilde{x}_{0})|\right)\, h(\kappa,H)\\
&\qquad\qquad\leq  C\,(\mathcal N^b+\mathcal N^n)\kappa^{2}\delta(\kappa)^{2}\hat{f}\left(\displaystyle\frac{H}{\kappa}|B_{0}(\widetilde{x}_{0})|\right)h(\kappa,H)+\mathcal N^n\kappa^{2}\delta(\kappa)^{2}\hat{f}\left(\displaystyle\frac{H}{\kappa}|B_{0}(\widetilde{x}_{0})|\right)h(\kappa,H).
\end{align*}
We divide both sides by $\kappa^{2}\delta(\kappa)^{2}\hat{f}\left(\displaystyle\frac{H}{\kappa}|B_{0}(\widetilde{x}_{0})|\right)\,h(\kappa,H)^{\frac{1}{2}}$ to get \eqref{eq-N'=o(N)}.
\end{proof}
%RRRRRRRRRRRRRRRRRRRRRRRRRRRRRRRRRRRRRRRRRRRRRRRRRRRRRRRRRRRRRRRRRRRRRRRRRRRRRRRRRRRRRRRRRRRRRRRRRRRRRRRRRRRRRR%
\begin{rem}
Using \eqref{app-ha} and \eqref{cond-H}, we obtain uniformly
\begin{equation}\label{rem-N'=o(N)}
\mathcal N^b\ll\mathcal N^n\,,\qquad{\rm as}~\kappa\longrightarrow+\infty\,.
\end{equation}
More precisely, we mean that $\mathcal N^b =\mathcal N^n\,\textit{e}(\kappa,H,\ell,\widetilde{x}_{0},x_{0})$, with $\textit{e}(\kappa,H,\ell,\widetilde{x}_{0},x_{0})$ is uniformly $\textit{o}(1)$ for any $\kappa\geq\kappa_{0}$, any $\rho$-admissible triple $(\ell,\widetilde{x}_{0},x_{0})$, any $H$ satisfying \eqref{cond-H}.
\end{rem}

\subsection{The results of Sandier-Serfaty.} 
Now we recall an important result of Sandier-Serfaty \cite{SS3}. Define the energy of $(u,A)\in H^{1}(\mathcal{D};\C)\times H^{1}_{\rm div}(\mathcal{D};\R)$ in a domain $\mathcal{D}\subset\R^{2}$ as follows
\begin{equation}\label{JK}
J_{\mathcal{D}}(u,A)=\int_{\mathcal{D}}|(\nabla-iA)u|^{2}+\frac{\kappa^{2}}{2}(1-|u|^{2})^{2}+|\curl A-h_{\rm ex}|^{2}\,dx\,.
\end{equation}
The next proposition is essentially proved\footnote{We replaced $\varepsilon$ by $\frac{1}{\kappa}$. We can indeed verify that only the upper bound of $J_{K}(u,A)$ is needed with no additional condition on $f(\varepsilon)$ and that the $\textit{o}(1)$ are actually uniform under uniform assumptions. Note also that we do not use in this proposition that $(u,A)$ is a critical point of $J_{\Omega}$.} in \cite[Proposition~5.1]{SS3}.
%%%%%%%%%%%%%%%%%%%%%%%%%%%%%%%%%%PPPPPPPPPPPPPPPPPP%%%%%%%%%%%%%%%%%%%%%%%%%%%%%%%%%%%%%%%%%%%%%%%%%%%%%%%%
\begin{prop}\label{serfaty}
Let $\hat{h}: (0,+\infty) \longrightarrow (0,+\infty)$ such that $\lim_{t\longrightarrow+\infty} \hat{h}(t)=0$, there exist two functions $s_{1},\,s_{2}:(0,+\infty) \longrightarrow (0,+\infty)$ satisfying
\begin{equation}\label{lim}
\lim_{t\longrightarrow+\infty} s_{1}(t)=0\,,\qquad\lim_{t\longrightarrow+\infty} s_{2}(t)=0\,.
\end{equation}
 Assume that $h_{\rm ex}$ is a function of $\kappa$ and $K$ is a square of side length $\gamma(\kappa)$ such that
\begin{equation}\label{cond-hex}
|\ln\kappa|\ll h_{\rm ex}\ll\kappa^{2}\qquad{\rm and}\qquad\ln\frac{\kappa}{\sqrt{h_{\rm ex}}}\ll h_{\rm ex}\gamma(\kappa)^{2}\ll {\rm min}\left(h_{\rm ex},\left(\ln\frac{\kappa}{\sqrt{h_{\rm ex}}}\right)^{2}\right)\,,\quad{\rm as}~ \kappa\longrightarrow \infty\,.
\end{equation}
If $(u,A)\in C^{1}(\overline{K};\C)\times C^{1}(\overline{K};\R^{2})$ verifies 
\begin{equation}\label{cond-hex2}
J_{K}(u,A)\leq h_{\rm ex}\gamma(\kappa)^{2}\ln\frac{\kappa}{\sqrt{h_{\rm ex}}}(1+\hat{h}(\kappa))\,,
\end{equation}
then, there exist  disjoint disks $\left(D(a_{i},r_{i})\right)_{i=1}^{m}$ such that:
\begin{enumerate}
\item $\sum r_{i}\leq h_{\rm ex}^{-\frac{1}{2}}$
\item $|u|>\frac{1}{2}$ on $\cup_{i}\partial D(a_{i},r_{i})$
\item If $d_{i}={\rm deg}\left(\frac{u}{|u|}, \partial D(a_{i},r_{i})\right)$\,, then, as $\kappa\longrightarrow+\infty$
\end{enumerate}
\begin{equation}
2\pi\sum_{i=1}^{m}d_{i}\geq h_{\rm ex}\gamma(\kappa)^{2}(1-s_{1}(\kappa))\qquad{\rm and}\qquad 2\pi\sum_{i=1}^{m}|d_{i}|\leq h_{\rm ex}\gamma(\kappa)^{2}(1+s_{2}(\kappa))\,.
\end{equation}
\end{prop}

The next lemma will give us that $\delta(\kappa)$, the side length of the square $Q^{j}_{\delta(\kappa)}$, satisfies \eqref{cond-hex} and will be useful in Proposition~\ref{vortices3}.
%LLLLLLLLLLLLLLLLLLLLLLLLLLLLLLLLLLLLLLLLLLLLLLLLLLLLLLLLLLLLLLLLLLLLLLLLLLLLLLLLLLLLLLLLLLLLLLLLLLLLLLLLLLLLLLLLLLLLLLLLLLLLLLLLLLLLL%
\begin{lem}\label{lm-cond-ser}
Under the assumptions of the previous subsection we have
\begin{align}
&\delta(\kappa)^{2}=\frac{1}{\varepsilon_{1}(\kappa,\ell,x_{0},\widetilde{x}_{0},H)}\,\,\frac{1}{\kappa H\,|B_{0}(\widetilde{x}_{0})|}\ln\frac{\kappa}{H\,|B_{0}(\widetilde{x}_{0})|}\,,\quad\qquad{as}~\kappa\longrightarrow+\infty\label{cond-ser1}\\
&\delta(\kappa)^{2}=\varepsilon_{2}(\kappa,\ell,x_{0},\widetilde{x}_{0},H)\,\,\frac{1}{\kappa H\,|B_{0}(\widetilde{x}_{0})|}\left(\ln\frac{\kappa}{H\,|B_{0}(\widetilde{x}_{0})|}\right)^{2}\,,\,\quad{as}~\kappa\longrightarrow+\infty\label{cond-ser2}\,,
\end{align}
where $\varepsilon_{1}(\kappa,\ell,x_{0},\widetilde{x}_{0},H)$ and $\varepsilon_{2}(\kappa,\ell,x_{0},\widetilde{x}_{0},H)$ are uniformly $\textit{o}(1)$ as $\kappa\longrightarrow+\infty$.
\end{lem}
\begin{proof}~\\
\textbf{Proof of \eqref{cond-ser1}}. We know that for all $\rho$-admissible triple $(\ell, x_{0}, \widetilde{x}_{0})$
$$\frac{1}{\kappa H |B_{0}(\widetilde{x}_{0})|}\ln\frac{\kappa}{H |B_{0}(\widetilde{x}_{0})|}\leq \frac{1}{\kappa H \rho}\ln\frac{\kappa}{H \rho}\leq \frac{3}{2}\,\frac{1}{\kappa H}\left(\ln\frac{\kappa}{H}\right)^{\frac{3}{2}}\,,$$
Also, we know that 
$$
\frac{1}{\kappa H}\left(\ln\frac{\kappa}{H}\right)^{\frac{3}{2}}\ll 2^{-\frac{7}{4}}\,\frac{1}{\kappa H}\left(\ln\frac{\kappa}{H}\right)^{\frac{7}{4}}\sim\delta(\kappa)^{2}\,.
$$
\textbf{Proof of \eqref{cond-ser2}}. On the other hand, we have
$$\frac{1}{\kappa H |B_{0}(\widetilde{x}_{0})|}\left(\ln\frac{\kappa}{H |B_{0}(\widetilde{x}_{0})|}\right)^{2}\geq \frac{1}{\kappa H \overline{\beta}_{0}}\left(\ln\frac{\kappa}{H \overline{\beta}_{0}}\right)^{2}\geq C\,\frac{1}{\kappa H}\left(\ln\frac{\kappa}{H}\right)^{2}\,,
$$
where $C$ is a positive constant and $\overline{\beta}_{0}$ is introduced in \eqref{beta0}.\\
It is clear that
$$
\frac{1}{\kappa H}\left(\ln\frac{\kappa}{H}\right)^{2}\gg 2^{-\frac{7}{4}}\,\frac{1}{\kappa H}\left(\ln\frac{\kappa}{H}\right)^{\frac{7}{4}}\sim\delta(\kappa)^{2}\,.
$$
\end{proof}

We can  prove the following result regarding the vortices of the minimizers in the `\textit{nice squares}'. We start with the admissible  squares contained in $\Omega\cap\{B_{0}>0\}$.
%PPPPPPPPPPPPPPPPPPPPPPPPPPPPPPPPPPPPPPPPPPPPPPPPPPPPPPPPPPPPPPPPPPPPPPPPPPPPPPPPPPPPPPPPPPPPPPPPPPPPPPPPPPPPPPPPPPPPPPPPPPPPPPPP%
\begin{prop}\label{vortices3}
Under Assumptions~\eqref{B(x)} and \eqref{cond-H} there exists $s_{1},\,s_{2}:(0, +\infty)\longrightarrow(0,+\infty)$ two functions satisfying \eqref{lim} and such that,  for any $(\ell, x_{0}, \widetilde{x}_{0})$ such that $\overline{Q^{j}_{\delta(\kappa)}}\subset\Omega\cap\{B_{0}>\rho\}$ and $\widetilde{x}_{0}\in \overline{Q^{j}_{\delta(\kappa)}}$ for which   $Q^{j}_{\delta(\kappa)}$ is a nice square, and any minimizer $(\psi,\Ab)\in H^{1}(\Omega,\C)\times H^{1}_{\rm div}(\Omega,\R^{2})$ of \eqref{eq-2D-GLf}, there exist disjoint disks $(D(a_{i,j},r_{i,j}))_{i=1}^{m_{j}}$ in $Q^{j}_{\delta(\kappa)}$ such that
\begin{equation}\label{cond-vortices3}
\begin{array}{lll}
 \bullet\quad\sum_{i=1}^{m_{j}} r_{i,j}\leq(\kappa H B_{0}(\widetilde{x}_{0}))^{-\frac{1}{2}}\\
 \bullet\quad|\psi|>\frac{1}{2}~\text{ on}~ \cup_{j}\partial D(a_{i,j},r_{i,j})\\
 \bullet\quad\text{If}~ d_{i,j}~ \text{ is the winding number of} ~\frac{\psi}{|\psi|}~ \text{ restricted to}~ \partial D(a_{i,j},r_{i,j}),~ \text{then}
\end{array}
\end{equation}
\begin{equation}\label{winding}
 2\pi\sum_{i=1}^{m_{j}}d_{i,j}\geq\delta(\kappa)^{2}\kappa H B_{0}(\widetilde{x}_{0})(1- s_{1}(\kappa))\,,\qquad\text{as}~\kappa\longrightarrow+\infty\,,
\end{equation}
and
\begin{equation}\label{winding1}
  2\pi\sum_{i=1}^{m_{j}}|d_{i,j}|\leq \delta(\kappa)^{2}\kappa H B_{0}(\widetilde{x}_{0})(1+s_{2}(\kappa))\,,\qquad\text{as}~\kappa\longrightarrow+\infty\,.
\end{equation}
\end{prop}
\begin{proof}~\\
We will  apply Proposition~\ref{serfaty}  with
\begin{equation}\label{choice-uA}
K= Q^{j}_{\delta(\kappa)},\,\gamma(\kappa)=\delta(\kappa),\, h_{\rm ex}=\kappa H B_{0}(\widetilde x_0),\, u=e^{-i\kappa H \varphi} \psi~{\rm and}~A(x)=\kappa H\,B_{0}(\widetilde{x}_{0})\,\Ab_{0}(x-x_0)\,,
\end{equation}
where $\Ab_{0}$ is the magnetic potential introduced in \eqref{eq-hc2-mpA0} and 
 $\varphi=\phi_{x_{0}}+\varphi_{x_{0},\widetilde{x}_{0}}$, with $\phi_{x_{0}}$  defined in \eqref{alpha} and $\varphi_{x_{0},\widetilde{x}_{0}}$ in  \eqref{F-A}.\\
Let us verify that the conditions of the proposition are satisfied for this choice.\\
First, we start by proving \eqref{cond-hex2}. 
Since $\curl\Ab_{0}=1$, then,
$$\curl A=\kappa H\,B_{0}(\widetilde{x}_{0})=h_{\rm ex}\,,$$
and consequently
\begin{equation}\label{A=hex}
\int_{Q_{\delta(\kappa)}^{j}}|\curl A-h_{\rm ex}|^{2}\,dx=0\,.
\end{equation}
This implies that, for any $j\in\mathcal J^n $
\begin{align}\label{J=E0}
J_{K}(u,A)&=\int_{Q^{j}_{\delta(\kappa)}}\left(|(\nabla-i\kappa H(B_{0}(\widetilde{x}_{0})\Ab_{0}(x-x_0)+\nabla\varphi))\psi|^{2}+\frac{\kappa^{2}}{2}\left(1-|\psi|^{2}\right)^{2}\right)\,dx\nonumber\\
&=\mathcal{E}_{0}(\psi,B_{0}(\widetilde{x}_{0})\Ab_{0}(x-x_0)+\nabla\varphi,Q^{j}_{\delta(\kappa)})\,.
\end{align}
Since $Q^{j}_{\delta(\kappa)}$ is a nice square, then,
\begin{equation}\label{lw-J0}
J_{K}(u,A)\leq\,\delta(\kappa)^{2}\,\kappa^{2}\hat{f}\left(\frac{H}{\kappa}\,B_{0}(\widetilde{x}_{0})\right)(1+h(\kappa,H)^{ \frac{1}{2}})\,,\qquad\text{as}~\kappa\longrightarrow+\infty\,.
\end{equation}
As a consequence of \eqref{main-f(b)}, \eqref{lw-J0} becomes
\begin{equation}\label{lw-J}
J_{K}(u,A)\leq\,\frac{1}{2}\,\delta(\kappa)^{2}\kappa H B_{0}(\widetilde{x}_{0})\,\ln\frac{\kappa}{H B_{0}(\widetilde{x}_{0})}(1+\hat{h}(\kappa,H))\,,\qquad\text{as}~\kappa\longrightarrow+\infty\,,
\end{equation}
where 
$$
\hat{h}(\kappa,H)=h(\kappa,H)^{\frac{1}{2}}+\hat{s}\left(\frac{H}{\kappa}\right)+\hat{s}\left(\frac{H}{\kappa}\right)\,h(\kappa,H)^{\frac{1}{2}}\,.$$
Notice that the function $\hat{h}(\kappa,H)$ is uniformly $\textit{o}(1)$ as $\kappa\longrightarrow+\infty$ and $H$ satisfying \eqref{cond-H}.\\
Secondly we prove \eqref{cond-hex}. In fact, under Assumption~\eqref{cond-H}, we can easly prove, uniformly as $\kappa\longrightarrow+\infty$,
$$
|\ln\kappa|\ll h_{\rm ex}\ll\kappa^{2}\,,
$$
and
$$
{\rm min}\left(h_{\rm ex},\left(\ln\frac{\kappa}{\sqrt{h_{\rm ex}}}\right)^{2}\right)=\left(\ln\frac{\kappa}{\sqrt{h_{\rm ex}}}\right)^{2}\,.
$$
Thanks to Lemma~\ref{lm-cond-ser}, we get that \eqref{cond-hex} is satisfied and in this way  we achieve the proof of Proposition~\ref{vortices3}.
\end{proof}

In light of Lemma~\ref{N'=o(N)}, we deduce from Proposition~\ref{vortices3} the distribution of vortices in a $\rho$-admissible  square $Q_{\ell}$. %PPPPPPPPPPPPPPPPPPPPPPPPPPPPPPPPPPPPPPPPPPPPPPPPPPPPPPPPPPPPPPPPPPPPPPPPPPPPPPPPPPPPPPPPPPPPPPPPPPPPPPPPPPPPPPPPPPPPPPPPPPPPPPPP%
\begin{prop}\label{vortices4}
Suppose that Assumptions~\eqref{B(x)} and \eqref{cond-H} are true. There exists two functions $s_{1},\,s_{2} :(0;+\infty)\longrightarrow(0;+\infty)$ satisfying \eqref{lim} and the following is true. Let $(\psi,\Ab)\in H^{1}(\Omega,\C)\times H^{1}_{\rm div}(\Omega,\R^{2})$ be a minimizer of \eqref{eq-2D-GLf} and $(\ell,x_0)$ such that $\overline{Q_{\ell}(x_{0})}\subset\Omega\cap\{B_{0}>\rho\}$. There exist a family of disjoint disks (indexed by $\mathcal{K}=\mathcal{K}_{\ell,x_{0}}$)  $(D(\widetilde{a}_{k},\widetilde{r}_{k}))_{k\in \mathcal{K}}$ in $Q_{\ell}(x_{0})$ such that
\begin{align}
&\bullet\quad \sum_{k\in\mathcal{K}}\widetilde{r}_{k}\,\leq (\kappa H B_{0}(\widetilde{x}_{0}))^{-\frac{1}{2}}\left(\frac{\ell}{\delta(\kappa)}\right)^{2}\,(1+\textit{o}(1))\,,\qquad{\rm as}~\kappa\longrightarrow+\infty \label{cond-vortices4}\\
& \bullet\quad|\psi|>\frac{1}{2}~\text{ on}~\cup_{k}\partial D(\widetilde{a}_{k},\widetilde{r}_{k})\label{winding4}\\
& \bullet \text{If}~\widetilde{d}_{k}~\text{is the winding number of}~\frac{\psi}{|\psi|}~\text{restricted to}~\partial D(\widetilde{a}_{k},\widetilde{r}_{k}),~\text{then as}~\kappa\longrightarrow+\infty\nonumber
\end{align}
\begin{align}\label{winding2}
2\pi\sum_{k\in \mathcal{K}}\widetilde{d}_{k}\geq\ell^{2}\kappa H\,{ B_{0}(\widetilde{x}_{0})\left(1-s_{1}(\kappa)\right)}\quad{\rm and}\quad 2\pi\sum_{k\in\mathcal{K}}|\widetilde{d}_{k}|\leq\ell^{2}\kappa H\,{ B_{0}(\widetilde{x}_{0})\left(1+\,s_{2}(\kappa)\right)}\,.
\end{align}
Here, the function $\textit{o}(1)$ is bounded independently of the choice of $\widetilde{x}_{0}$ and the minimizer $(\psi,\Ab)$.
\end{prop} 

\begin{proof}
Recall that $Q_{\ell}(x_0)$ is decomposed into $\mathcal N^n$   `\textit{nice squares}' $(Q^{j}_{\delta(k)})_{j\in\mathcal J^n }$ and $\mathcal N^b$  `\textit{bad squares}' $(Q^{j}_{\delta(k)})_{j\in\mathcal J^b }\,$.\\
 In every nice square $Q^{j}_{\delta(\kappa)}$, Proposition~\ref{vortices3} tells us that there exist disjoint disks $ (D(a_{i,j},r_{i,j}))_{i=1}^{m_{j}}$ such that \eqref{cond-vortices3}, \eqref{winding} and \eqref{winding1} hold. Let $ (D(\widetilde{a}_{k},\widetilde{r}_{k}))_{k\in\mathcal{K}}=\left( D(a_{i,j},r_{i,j})\right)_{i,j}$ be the family of disjoint disks in $\cup_{j\in\mathcal J^n }Q^{j}_{\delta(\kappa)}$. Clearly $$\sum_{k\in \mathcal{K}}\widetilde{r}_{k}=\sum_{j\in\mathcal J^n } \sum_{i=1}^{m_{j}}r_{i,j}\,.$$
This implies that
\begin{align*}
\sum_{k\in \mathcal{K}}\widetilde{r}_{k}\leq (\kappa H B_{0}(\widetilde{x}_{0}))^{-\frac{1}{2}}\mathcal N^n\,.
\end{align*}
Having in mind \eqref{area-square} and \eqref{rem-N'=o(N)}\,, we have, as $\kappa\longrightarrow+\infty\,$, 
\begin{equation}\label{Ndelta=L}
\mathcal N^n\,\frac{\delta(\kappa)^{2}}{\ell^{2}}\longrightarrow 1\,.
\end{equation}
This implies that 
$$
\sum_{k\in \mathcal{K}}\widetilde{r}_{k}\leq(\kappa H B_{0}(\widetilde{x}_{0}))^{-\frac{1}{2}}\left(\frac{\ell}{\delta(\kappa)}\right)^{2}(1+\textit{o}(1))\,,
$$
where the function $\textit{o}(1)$ is bounded independently of the choice of $\widetilde{x}_{0}$.\\
Let $\widetilde{d}_{k}$ be the winding number of $\frac{\psi}{|\psi|}$ restricted to $\partial D(\widetilde{a}_{k},\widetilde{r}_{k})\,$, then, for any $k\in\mathcal{K}$
\begin{align*}
\sum_{k\in \mathcal{K}}\widetilde{d}_{k}=\sum_{j\in\mathcal J^n } \sum_{i=1}^{m_{j}} d_{i,j}\,.
\end{align*}
Since from \eqref{winding}, \eqref{winding1} and \eqref{Ndelta=L}, we get, $\text{as}~\kappa\longrightarrow+\infty\,,$
\begin{align}\label{winding3}
 2\pi \sum_{k\in\mathcal{K}}\widetilde{d}_{k}= 2\pi\sum_{j\in\mathcal J^n }\sum_{i=1}^{m_{j}}d_{i,j}&\geq\mathcal N^n\delta(\kappa)^{2}\kappa H  B_{0}(\widetilde{x}_{0}) \left(1-s_{1}(\kappa)\right)  \nonumber\\
&\geq\ell^{2}\kappa H  B_{0}(\widetilde{x}_{0}) \left(1-s_{1}(\kappa)\right) \,, 
\end{align}
 and
\begin{align}\label{winding33}
2\pi \sum_{k\in\mathcal{K}}|\widetilde{d}_{k}|= 2\pi\sum_{j\in\mathcal J^n }\sum_{i=1}^{m_{j}}|d_{i,j}|&\leq\mathcal N^n\delta(\kappa)^{2}\kappa H  B_{0}(\widetilde{x}_{0}) \left(1+s_{2}(\kappa)\right) \nonumber\\
&\leq\ell^{2}\kappa H  B_{0}(\widetilde{x}_{0}) \left(1+s_{2}(\kappa)\right)\,.
\end{align}
This finishes the proof of Proposition~\ref{vortices4}.
\end{proof}
%%%%%%%%%%%%%%%%%%%%%%%%%%%%%%%%%%%%%%%proof of theorem%%%%%%%%%%%%%%%%%%%%%%%%%%%%%%%%%%%%%%%%%%%%%%%%%%%%%%
\subsection{Proof of Theorem~\ref{distribution-vortices}}
Let $(\psi,\Ab)\in H^{1}(\Omega,\C)\times H^{1}_{\rm div}(\Omega,\R^{2})$ be a minimizer of \eqref{eq-2D-GLf} and $\Gamma_{\ell}:=\ell\Z\times\ell\Z$ a lattice of $\R^{2}$. For all $\gamma\in\Gamma_{\ell}$, we consider the family of squares  $Q_{\ell}(\gamma)$ and $\widetilde{\gamma}\in Q_{\ell}(\gamma)$. Consider an open set $S\subset \Omega\cap\{B_{0}>0\}$ such that the boundary of $S$ is smooth. Let
\begin{equation}
\mathcal{J}_{\ell}=\{\gamma\,;\quad\overline{Q_{\ell}(\gamma)}\subset S\cap\{B_0>\rho\}\, \}\,,
\end{equation}
\begin{equation}
\mathcal{M}=\text{card}\,\mathcal{J}_{\ell}\,,
\end{equation}
and
\begin{equation}
S_{\ell}=\text{int}\,\left(\cup_{\gamma\in\mathcal{J}_{\ell}} \overline{Q_{\ell}(\gamma)}\right)\,.
\end{equation}
Then, as $\kappa\longrightarrow+\infty$, we have
\begin{equation}\label{cardinal}
\mathcal{M}\times\ell^{2}\longrightarrow |S|\,.
\end{equation}
\textbf{Proof of \eqref{p2}:}~\\
Proposition~\ref{vortices4} tells us that there exist disjoint disks $ (D(\widetilde{a}_{k,\gamma},\widetilde{r}_{k,\gamma}))_{k\in \mathcal{K}_{\ell,\gamma}}$ in each square $Q_{\ell}(\gamma)$ with $\overline{Q_{\ell}(\gamma)}\subset\Omega\cap\{B_{0}>\rho\}$ such that \eqref{cond-vortices4} and \eqref{winding2} hold. We introduce the measure by
\begin{equation}
\mu_{\kappa}:=\frac{2\pi}{\kappa H}\sum_{\gamma\in\mathcal{J}_{\ell}}\sum_{k\in\mathcal{K}_{\ell,\gamma}}\widetilde{d}_{k,\gamma}\,\delta_{\widetilde{a}_{k,\gamma}}\,,
\end{equation}
where $\widetilde{d}_{k,\gamma}$ is the winding number introduced before \eqref{winding2} and  $\delta_{\widetilde{a}_{k,\gamma}}$ is the unit Dirac mass at $\widetilde{a}_{k,\gamma}$.\\
Having in mind \eqref{winding2} we have for any $(\ell,\gamma,\widetilde{\gamma})$ such that $\overline{Q_{\ell}(\gamma)}\subset\Omega\cap\{B_{0}>\rho\}$ and $\widetilde{\gamma}\in \overline{Q_{\ell}(\gamma)}$
$$
B_{0}(\widetilde{\gamma})\ell^{2}(1-s_{1}(\kappa))\leq\frac{2\pi}{\kappa H}\sum_{k\in\mathcal{K}_{\ell,\gamma}}\widetilde{d}_{k,\gamma}\leq B_{0}(\widetilde{\gamma})\ell^{2}(1+s_{2}(\kappa))\,.
$$
Using \eqref{cardinal}, we obtain
\begin{equation}\label{1st}
\left(\sum_{\gamma\in \mathcal{J}_{\ell}} B_{0}(\widetilde{\gamma})\ell^{2}\right)-2\,\overline{\beta}_{0}|S|\,s_{1}(\kappa)\leq\frac{2\pi}{\kappa H}\sum_{\gamma\in\mathcal{J}_{\ell}}\sum_{k\in\mathcal{K}_{\ell,\gamma}}\widetilde{d}_{k,\gamma}\leq \left(\sum_{\gamma\in \mathcal{J}_{\ell}} B_{0}(\widetilde{\gamma})\ell^{2}\right)+2\,\overline{\beta}_{0}|S|\,s_{2}(\kappa)\,.
\end{equation}
Here, we have used the fact that $B_{0}(\widetilde{\gamma})\leq\overline{\beta}_{0}$ to estimate the errors terms, where $\overline{\beta}_{0}$ is introduced in \eqref{beta0}.\\
Now it is time to determine $\sum_{\gamma\in \mathcal{J}_{\ell}} B_{0}(\widetilde{\gamma})\ell^{2}$. We will do this in two steps:\\
\textbf{Upper bound:} Notice that till now $\widetilde{\gamma}$ was an arbitrary point in $Q_{\ell}(\gamma)$, but that our estimates are independent of this choice. We now select $\widetilde{\gamma}\in \overline{Q_{\ell}(\gamma)}$ such that $B_{0}(\widetilde{\gamma})=\underline{B}_{\gamma,\ell}$ with $\underline{B}_{\gamma,\ell}$ satisfying \eqref{inf-B}, and get 
 $$\sum_{\gamma\in \mathcal{J}_{\ell}} B_{0}(\widetilde{\gamma})\ell^{2}=\sum_{\gamma\in \mathcal{J}_{\ell}}\underline{B}_{\gamma,\ell}\ell^{2}\,.$$
We recognize in the right hand side above the lower Riemann sum of $x\longrightarrow B_{0}(x)$ and we use that $S_{\ell}\subset S$ to obtain
 \begin{equation}\label{up-B0}
\sum_{\gamma\in \mathcal{J}_{\ell}} B_{0}(\widetilde{\gamma})\ell^{2}\leq\int_{S} B_{0}(x)\,dx\,.
 \end{equation}
\textbf{Lower bound:} We select $\widetilde{\gamma}\in Q_{\ell}(\gamma)$ such that $B_{0}(\widetilde{\gamma})=\overline{B}_{\gamma,\ell}$ with $\overline{B}_{\gamma,\ell}$ satisfies \eqref{supB}. Similarly to what we did in the upper bound above, we get
$$
\sum_{\gamma\in \mathcal{J}_{\ell}}B_{0}(\widetilde{\gamma})\ell^{2}\geq\int_{S_{\ell}} B_{0}(x)\,dx\,.
$$
Notice that using the regularity of $\partial S$ and \eqref{B(x)}, we have as $\kappa\longrightarrow+\infty$
\begin{equation}\label{area-2nd}
|S\setminus S_{\ell}|=\mathcal{O}(\ell|\partial S|)\,.
\end{equation}
Therefore
\begin{align*}
\int_{S_{\ell}} B_{0}(x)\,dx&= \int_{S} B_{0}(x)\,dx-\int_{S\setminus S_{\ell}} B_{0}(x)\,dx\\
&\geq \int_{S} B_{0}(x)\,dx-\overline{\beta}_{0}|S\setminus S_{\ell}|\\
&\geq \int_{S} B_{0}(x)\,dx-C\,\ell\,,
\end{align*}
where $\overline{\beta}_{0}$ is introduced in \eqref{beta0} and $C$ is a positive constant.\\
This implies that
\begin{equation}\label{lw-B0}
\sum_{\gamma\in \mathcal{J}_{\ell}} B_{0}(\widetilde{\gamma})\ell^{2}\geq\int_{S} B_{0}(x)\,dx-C\,\ell\,.
\end{equation}
The estimates in \eqref{up-B0} and \eqref{lw-B0} allow us to deduce from \eqref{1st} that
\begin{equation}
-C\,\ell-2\,\overline{\beta}_{0}|S|\,s_{1}(\kappa)\leq\mu_{\kappa}(S)-\int_{S} B_{0}(x)\,dx\leq+2\,\overline{\beta}_{0}|S|\,s_{2}(\kappa)\,.
\end{equation}
Consequently, as $\kappa\longrightarrow+\infty$
\begin{equation}\label{conv}
\mu_{\kappa}(S) \longrightarrow \int_{S} B_{0}(x)\,dx\,,\qquad\forall S\subset\Omega\cap\{B_{0}>0\}\,.
\end{equation}
In light of \eqref{conv}, we can easily show that $\mu_{\kappa}$ converge weakly to $\mu=B_{0}(x)\,dx$, which means that:
$$
\mu_{\kappa}(f)\longrightarrow\mu(f)\,,\qquad\forall f\in C_{0}(\Omega\cap\{B_{0}>0\})\,.
$$
\textbf{Proof of \eqref{p1}:} We will prove that the sum of the radii of the disks $ (D(\widetilde{a}_{k,\gamma},\widetilde{r}_{k,\gamma}))_{k\in \mathcal{K}_{\ell,\gamma}\,\gamma\in\mathcal{J}_{\ell}}$ is less than 
$$(\kappa H)^{\frac{1}{2}}\,\left(\ln\frac{\kappa}{H}\right)^{-\frac{7}{4}}\,\int_{S}\frac{1}{\sqrt{B_{0}(x)}}\,dx\,(1+\textit{o}(1))\,.$$
In fact, remembering the choice of $\delta(\kappa)$ in \eqref{delta(kappa)}, \eqref{cardinal} and that $\overline{Q_{\ell}(\gamma)}\subset\Omega\cap\{B_{0}>0\}$, we have
\begin{align}
\sum_{\gamma\in\mathcal{J}_{\ell}} \widetilde{r}_{k,\gamma}&=\sum_{\gamma\in\mathcal{J}_{\ell}}\sum_{k\in\mathcal{K}}\widetilde{r}_{k}\nonumber\\
&\leq (\kappa H)^{\frac{1}{2}}\,\left(\ln\frac{\kappa}{H}\right)^{-\frac{7}{4}}\,\sum_{\gamma\in\mathcal{J}_{\ell}}\frac{1}{\sqrt{B_{0}(\widetilde{\gamma})}}\,\ell^{2}\,(1+\textit{o}(1))\,.\label{sum-r}
\end{align}
We select $\widetilde{\gamma}\in \overline{Q_{\ell}(\gamma)}$ such that
$$
\frac{1}{\sqrt{B_{0}(\widetilde{\gamma})}}=\inf_{\widehat{\gamma}\in Q_{\gamma,\ell}} \frac{1}{\sqrt{B_{0}(\widehat{\gamma})}}\,,
$$
and we recognize in the right hand side of \eqref{sum-r} the lower Riemann sum of $x\longrightarrow\frac{1}{\sqrt{B_{0}(x)}}$, we get
$$
\sum_{\gamma\in\mathcal{J}_{\ell}} \widetilde{r}_{k,\gamma}\leq (\kappa H)^{\frac{1}{2}}\,\left(\ln\frac{\kappa}{H}\right)^{-\frac{7}{4}}\,\int_{S}\frac{1}{\sqrt{B_{0}(x)}}\,dx\,(1+\textit{o}(1))\,.
$$
\textbf{End of the proof of Theorem~\ref{distribution-vortices}:}
In $\{B_{0}<0\}\cap\Omega$, we apply Proposition~\ref{serfaty}  with
\begin{equation}\label{choice-uA2}
K= Q^{j}_{\delta(\kappa)},\,\gamma(\kappa)=\delta(\kappa),\, h_{\rm ex}=-\kappa H B_{0}(\widetilde x_0),\, u=e^{i\kappa H \varphi} \overline{\psi}~{\rm and}~A(x)=-\kappa H\,B_{0}(\widetilde{x}_{0})\,\Ab_{0}(x-x_0)\,.
\end{equation}
So we get that, the convergence of mesure $\mu_{\kappa}$ in \eqref{conv} is still true when $S\subset\Omega\cap\{B_{0}<0\}$.\\
Similarly, we can control the convergence of $|\mu_{\kappa}|(S)$. Now we observe that the support of $\mu_{\kappa}$ does not meet $\{B_{0}=0\}$. Hence $\mu_{\kappa}(S)=\mu_{\kappa}(S\cap\{B_{0}<0\})+\mu_{\kappa}(S\cap\{B_{0}>0\})$ and we can apply the previous arguments to $S_{-}=S\cap\{B_{0}<0\}$ and $S_{+}=S\cap\{B_{0}>0\}$.

\section*{Acknowledgements}
This work is partially supported by a grant from Lebanese University.
I would like to thank my supervisors \textit{B.Helffer} and \textit{A.Kachmar} for their support, \textit{E.Sandier} for discussions around vortices and \textit{S.Fournais} for his help in the proof of Proposition~\ref{FR}.

\begin{figure}[ht!]
\begin{center}
\includegraphics[scale=1]{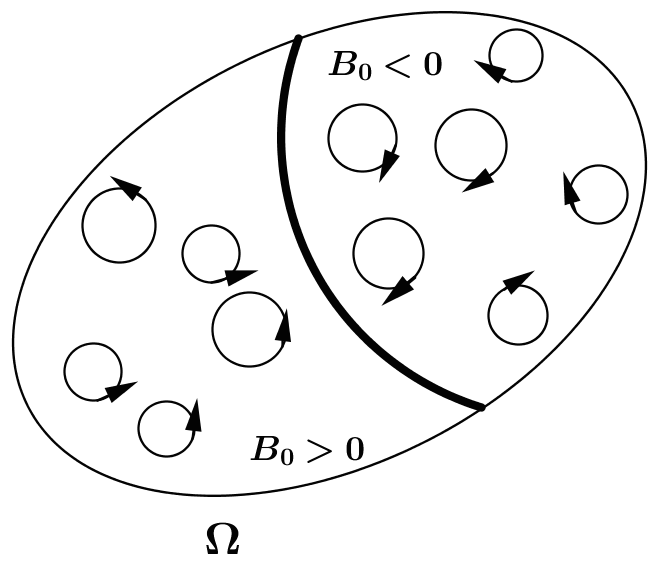} 
\caption{Vortices}
\end{center}
\end{figure}

\end{document}